\documentclass[11pt,a4paper]{article}

\usepackage[T1]{fontenc}
\usepackage[english]{babel}
\usepackage[utf8]{inputenc}
\usepackage[pdftex]{graphicx}           
\usepackage{setspace}                   
\usepackage{indentfirst}                
\usepackage{makeidx}                    
\usepackage{courier}                    
\usepackage{type1cm}                    
\usepackage{titletoc}
\usepackage[font=small,format=plain,labelfont=bf,up,textfont=it,up]{caption}
\usepackage[usenames,svgnames,dvipsnames]{xcolor}
\usepackage[a4paper,top=2.54cm,bottom=2.0cm,left=2.0cm,right=2.54cm]{geometry} 
\usepackage[pdftex,plainpages=false,pdfpagelabels,pagebackref,colorlinks=true,citecolor=DarkGreen,linkcolor=NavyBlue,urlcolor=DarkRed,bookmarksopen=true]{hyperref}
\usepackage[all]{hypcap} 
\fontsize{60}{62}\usefont{OT1}{cmr}{m}{n}{\selectfont}
\usepackage{dsfont}
\usepackage{amsmath}
\usepackage{amsthm,amssymb}
\usepackage{mathabx}

\usepackage{amsfonts}
\usepackage{relsize}
\usepackage{mathrsfs} 
\usepackage{pgf,tikz}
\usepackage{mathabx}
\usetikzlibrary{arrows}
\usepackage{enumitem}
\usepackage{genealogytree}

\usepackage{bbold}

\DeclareRobustCommand*{\cev}[1]{\reflectbox{\ensuremath{\vec{\reflectbox{\ensuremath{#1}}}}}}

\setlist[description]{font=\normalfont\itshape\textbullet\space}
\usepackage{fancyhdr}
\graphicspath{{./figuras/}}             
\makeindex                              
\fontsize{60}{62}\usefont{OT1}{cmr}{m}{n}{\selectfont}
\cleardoublepage
\normalsize

\def\R{{\mathbb R}}

\def\Z{{\mathbb Z}}
\def\Zd{\mathbb{Z}^d}
\def\N{{\mathbb N}}

\newcommand*\dif{\mathop{}\!\mathrm{d}}
\def\ds{\mathrm d s}

\def\fC{{\mathfrak C}}
\def\fQ{{\mathfrak Q}}
\def\fR{{\mathfrak R}}

\def\mt{{\mathfrak T}}
\def\mi{{\mathfrak I}}

\def\L{{\mathcal L}}
\def\cR{{\mathcal R}}
\def\cI{{\mathcal I}}
\def\cJ{{\mathcal J}}

\def\B{{\mathcal B}}
\def\cH{{\mathcal H}}
\def\cQ{{\mathcal Q}}

\def\cE{{\mathcal E}}
\def\G{{\mathcal G}}
\def\T{{\mathcal T}}
\def\W{{\mathcal W}}
\def\ss{{\mathcal S}}

\def\cD{{\cal D}}

\def\M{{\mathcal M}}

\def\Nx{{\mathcal N}_{\x}}
\def\cN{{\mathcal N}}
\def\F{{\cal F}}

\def\sC{\mathsf C}
\def\sd{\mathsf d}
\def\sfD{\mathsf D}

\def\H{\mathsf H}
\def\h{\mathsf h}

\def\sK{\mathsf K}
\def\sL{\mathsf L}

\def\sN{\mathsf N}
\def\sQ{\mathsf Q}
\def\sT{\mathsf T}

\def\sV{\mathsf V}
\def\w{\mathsf w}
\def\bw{\mathbf w}
\def\bt{\mathbf T}

\def\Y{\mathsf Y}

\def\sx{\mathsf x}
\def\cx{\sy}
\def\sy{\mathsf y}

\def\sz{\mathsf z}
\def\brex{\breve X}
\def\bro{\breve\o}
\def\oq{\o^{\text{eq}}}
\def\op{\o^+}

\def\vpq{\vp^{\text{eq}}}

\def\ozq{\o^\sz_{\text{eq}}}
\def\chozq{{\check\o}^\sz_{\text{eq}}}
\def\hozq{{\hat\o}^\sz_{\text{eq}}}

\def\choz{{\check\o}^\sz}

\def\brel{\breve\sL}
\def\bret{\breve\tau}
\def\circo{\mathring\o}

\def\tio{\tilde\o}

\def\cirx{\mathring X}
\def\cirxi{\mathring \xi}
\def\brexi{\breve \xi}
\def\cirt{\mathring\tau}
\def\circa{\mathring\mA}
\def\mA{\mathsf A}
\def\circh{\mathring\cH}
\def\cirm{\mathring\M}
\def\cirn{\mathring\cN}
\def\bh{\bar\cH}

\def\mvs{\mathring \varsigma}

\def\ve{\varepsilon}

\def\alt{A_{\ell,t}}
\def\blt{B_{\ell,t,\eps}}

\def\Pm{{\mathrm P}}
\def\P{{\mathbf P}}
\def\Pz{{\mathbf P}_{\zero}}
\def\Pr{{\mathbb P}}
\def\Em{{\mathrm E}}
\def\E{{\mathbf E}}
\def\Er{{\mathbb E}}

\def\ms{\mathsmaller}

\def\nn{\nonumber}

\def\x{\mathbf{x}}

\def\xn{\mathsf{x}_{\ms{n}}}

\def\xz{\mathsf{x}_{\ms{0}}}
\def\xu{\mathsf{x}_{\ms{1}}}
\def\y{\mathbf{y}}

\def\one{\mathbf 1}
\def\oone{\mathbb 1}
\def\zero{\mathbf 0}
\def\p{\mathbf p}
\def\pp{p^\psi}
\def\qp{q^\psi}

\def\q{\mathbf q}
\def\mzero{\ms{\ms{\zero}}}

\def\k{\mathbf{k}}

\def\o{\omega}
\def\fvt{\vec{\vartheta}}
\def\bvt{\cev{\vartheta}}

\def\cho{\check\omega}

\def\d{\delta}

\def\D{\Delta}

\def\Ups{\Upsilon}
\def\t{\tau}

\def\tn{\tau_{\ms{n}}}

\def\tu{\t_{\ms{1}}}

\def\s{\sigma}

\def\vf{\varphi}
\def\vs{\varsigma}
\def\vs{\varsigma}
\def\vr{\varrho}

\def\ga{\gamma}
\def\eps{\epsilon}
\def\tsi{\tilde\s}

\def\vp{\varpi}

\definecolor{title}{RGB}{255,0,90}
\definecolor{def}{rgb}{0.0, 0.28, 0.67}
\definecolor{bege}{rgb}{0.66, 0.637, 0.492}
\definecolor{bpink}{rgb}{0.95,0.82,0.69}
\definecolor{xdxdff}{rgb}{0.49,0.49,1}
\definecolor{qqtttt}{rgb}{0,0.2,0.2}
\definecolor{qqwwcc}{rgb}{0,0.4,0.8}
\definecolor{ffttww}{rgb}{1,0.2,0.4}
\definecolor{qqqqff}{rgb}{0,0,1}
\definecolor{zzttqq}{rgb}{0.6,0.2,0}
\definecolor{ttfftt}{rgb}{0.2,1,0.2}
\definecolor{ttffqq}{rgb}{0.2,1,0}
\definecolor{qqqqff}{rgb}{0,0,1}
\definecolor{qqwwff}{rgb}{0,0.4,1}
\definecolor{ttzzff}{rgb}{0.2,0.6,1}
\definecolor{qqffqq}{rgb}{0,1,0}
\definecolor{fftttt}{rgb}{1,0.2,0.2}
\definecolor{ffqqww}{rgb}{1,0,0.4}
\definecolor{xdxdff}{rgb}{0.49,0.49,1}
\definecolor{cqcqcq}{rgb}{0.75,0.75,0.75}

\newtheorem{teorema}{Theorem}[section]
\newtheorem{lema}[teorema]{Lemma}
\newtheorem{corolario}[teorema]{Corollary}
\newtheorem{observacao}[teorema]{Remark}
\newtheorem{proposicao}[teorema]{Proposition}

\author{Luiz Renato Fontes \footnote{IME-USP, Rua do Mat\~ao 1010, 05508-090
		S\~ao Paulo SP,  Brazil, lrfontes@usp.br} \thanks{Partially
		supported by CNPq grant 307884/2019-8, and FAPESP grants 2017/10555-0 and 2015/00053-2}
	\and Pablo A.~Gomes\footnote{IME-USP, Rua do Mat\~ao 1010, 05508-090
		S\~ao Paulo SP,  Brazil, pabloag7@yahoo.com.br} \thanks{Supported by	FAPESP grant 2020/02636-3}
	\and Maicon A.~Pinheiro \footnote{IME-USP, Rua do Mat\~ao 1010, 05508-090
		S\~ao Paulo SP,  Brazil, avertere@zoho.com} \thanks{Supported by CNPq and CAPES institutional fellowships}}

\title{Random walk in a birth-and-death dynamical environment} 

\date{}

\begin{document}
	
	\maketitle

	
	\begin{abstract}
		
		We consider a particle moving in continuous time as a Markov jump process; its discrete chain is given by an ordinary random walk on $\Zd$ , and its jump rate at $(\x,t)$ is given by a fixed function $\vf$ of the state of a birth-and-death (BD) process at $\x$ on time $t$; BD processes at different sites are independent and identically distributed, and $\vf$ is assumed non increasing and vanishing at infinity. We derive a LLN and a CLT for the particle position when the environment is 'strongly ergodic'. In the absence of a viable uniform lower bound for the jump rate, we resort instead to stochastic domination, as well as to a subadditive argument to control the time spent by the particle to give $n$ jumps; and we also impose conditions on the initial (product) environmental initial distribution. We also present results on the asymptotics of the environment seen by the particle (under different conditions on $\vf$).

	\end{abstract}


	\smallskip





\section{Introduction}
\label{intro}

In this paper, we analyse the long time behavior of  random walks taking place in {\em an evolving field of traps}. A starting motivation is 
to consider a {\em dynamical environment} version of Bouchaud's trap model on $\Zd$. In the (simplest version of the) latter model, we
have a continuous time random walk (whose embedded chain is an ordinary random walk) on $\Zd$ with spatially inhomogeneous jump rates,
given by a field of iid random variables, representing traps. The greater interest is for the case where the inverses of the rates
are heavy tailed, leading to subdiffusivity of the particle (performing the random walk), and to the appearance of the phenomenon of aging.
See~\cite{FIN02} and~\cite{BAC}.

In the present paper, we have again a continuous time random walk whose embedded chain is an ordinary random walk 
(with various hypotheses on its jump distribution, depending on the result), but now the rates are
spatially {\em as well as temporally} inhomogeneous, the rate at a given site and time is given by a (fixed) function, which we denote by $\vf$,
of the state of a birth-and-death chain (in continuous time; with time homogeneous jump rates) at that site and time;  birth-and-death chains for different sites are iid and ergodic. 

We should  not expect subdiffusivity if $\vf$ is bounded away from $0$, so we make the opposite assumption for our first main result, which is 
nevertheless a Central Limit Theorem for the position of the particle (so, no subdiffusivity there, either), as well as a corresponding 
Law of Large Numbers.

CLT's for random walks in dynamical random environments have been, from a  more general point of view, or under different motivations, 
previously established in a variety of situations; we mention~\cite{BMP97},
\cite{BZ06},~\cite{DKL08},~\cite{RV13} for a few cases with fairly general environments, and ~\cite{dHdS},~\cite{MV},~\cite{HKT} 
in the case of environments given by specific  interacting particle systems;~\cite{BMPZ} and~\cite{BPZ} deal with a case where the jump times of the particle are iid.
There is a relatively large literature establishing strong LLN's for the position of the particle in random walks in space-time random environments; 
besides most of the references given above, which also establish it, we mention~\cite{AdHR} and~\cite{BHT}. \cite{Y09} derives large deviations for 
the particle in the case of an iid space-time environment.

These papers assume (or have it naturally) in their environments an ellipticity condition, from which our environment crucially departs, 
in the sense of our jump rates not being bounded away from $0$. 
Jumps are generally also taken to be bounded, a possibly merely technical assumption in many respects,
which we in any case forgo. It should also be said that in many other respects, these models are quite more general, or more correlated
than ours\footnote{This is perhaps a good point to remark that even though our environment is constituted by iid Birth-and-Death processes, 
	and the embedded chain of the particle is independent of them, the continuous time motion of the particle brings about a correlation between
	the particle and the environment.}.

So, we seem to need a different approach, and that is what we develop here. Our argument requires monotonocity of $\vf$, 
and "strong enough" ergodicity of the environmental chains (translating into something like a second moment condition on its equilibrium 
distribution). 

The main building block for arguing our CLT, in the case where the initial environment is identically 0, is a Law of Large Numbers for the time
that the particle takes to make $n$ jumps;  this in turn relies on a subadditivity argument, resorting to the Subadditive Ergodic Theorem; 
in order to obtain the control the latter theorem requires on expected values, we rely on a domination of the environment left by the particle
at jump times (when starting from equilibrium); this is a {\em stochastic} domination, rather than a strong domination, which would be provided 
by the infimum of $\vf$, were it positive. We extend to more general, product initial environments, with a unifom exponentially decaying
tail (and also restricting in this case to spatially homogeneous environments), by means of coupling arguments.

We expect to be able to establish various forms of subdiffusivity in this model when the environment either is not ergodic or not "strongly ergodic"
(with, say, heavy tailed equilibrium measures). This is under current investigation. \cite{BPZ} has results in this direction in the 
case where the jump times of the particle are iid.

Another object of analysis in this paper is the long time behavior of the environment seen by the particle at jump times. We show convergence in distribution under different hypotheses (but always with spatially homogeneous environments, again in this case), and also that the limiting distribution is absolutely continuous with respect to the product of environmental equilibria. We could not bring the domination property mentioned above to bear for this result in the most involved instances
of a recurrent embedded chain, so we could not avoid the assumption of a bounded away from 0 $\vf$ (in which case, monotonicity can be dropped),
and a "brute force", strong tightness control this allows. 
This puts us back under the ellipticiy restriction on the rates\footnote{But the state space of the environment remains non compact.}, 
adopted in many results of the same nature that have been previously obtained, 
as in many of the above mentioned references, to which we add~\cite{BMP94}.

\begin{center}
	------------------------------------------
\end{center}

The remainder of this paper is organized as follows. In Section~\ref{mod} we define our model in detail, and discuss some of its properties.
Section~\ref{conv} is devoted to the formulations and proofs of the LLN and CLT under an environment started from the identically 0 configuration.
The main ingredient, as mentioned above, a LLN for the time that the particle takes to give $n$ jumps, is developed in Subsection~\ref{sec:3.1}, and the remaining subsections are devoted for the conclusion. In Section~\ref{ext} we extend the CLT for more general (product) initial configurations of the 
environment (with a uniform exponential moment). In Section~\ref{env} we formulate and prove our result concerning the environment seen from the
particle (at jump times). Three appendices are devoted for auxiliary results concerning birth-and-death processes and ordinary (discrete time) random walks.



\section{The model}
\label{mod}

\setcounter{equation}{0}

\noindent 
For $d\in \N_*:=\N\setminus \{0\}$ and $S\subset \R^d$, let $\cD \left(\R_+,S\right)$ denote the set of  
c\`adl\`ag trajectories from $\R_+$ to $S$. 
We represent by $\zero\in E$ and $\mathbf{1}\in E$, $E=\N^d, \Zd,
\N^{\Zd}$,  respectively, the null element, and the element with all coordinates identically equal to 1. 

We will use the notation $M\sim BDP(\p,\q)$ to to indicate that $M$ is a birth-and-death process on $\N$ 
with birth rates $\p=(p_n)_{n\in\N}$ and death rates $\q=(q_n)_{n\in\N_*}$. 
We will below consider indepent copies of such a process, and we will assume that
$p_n,q_n\in(0,1)$ for all $n$, $p_n+q_n\equiv1$ and
\begin{equation}\label{erg}
	\sum_{n\geq1}\prod_{i=1}^n\frac{p_{i-1}}{q_i}<\infty.
\end{equation}
This condition is well known to be equivalent to ergodicity of such a process.
We will also assume that $p_n\leq q_n$ for all $n$ and $\inf_np_n>0$. See Remark~\ref{relax} at the end of Section~\ref{conv}.

We now make an explicit construction of our process, namely, the random walk in a birth-and-death (BD) environment.
Let $\o=\left(\o_{\x} \right)_{\x\in \Zd}$ be an independent family of BDP's as prescribed in the 
paragraph of~\eqref{erg} above, each started from its respective initial distribution $\mu_{\x,0}$, independently of each other; 
we will denote by $\mu_{\x,t}$ the distribution of $\o_{\x}(t)$, 
$t\in\R_+$, $\x\in\Zd$; 
$\o$ plays the role of random dynamical environment of our random walk, which we may view as a 
stochastic process $\left(\o(t)\right)_{t\in\R_+}$ on $\Lambda:=\N^{\Zd}$ with initial distribution
$\hat{\mu}_0:=\bigotimes\limits_{\x\in\Zd}\mu_{\x,0}$  and trajectories living on $A:=\cD\left(\R_+,\N\right)^{\Zd}$. 
Let  $\Pm_{\hat{\mu}_{{0}}}$ denote the law of $\o$.

Let now $\pi$ be a probability on $\Zd\setminus\left\lbrace \zero\right\rbrace$, 
and let $\xi:=\left\lbrace \xi_n\right\rbrace_{n\in\N_*}$ be an iid sequence of random vectors taking values in $\Zd\setminus\left\lbrace \zero\right\rbrace$, each distributed as $\pi$; $\xi$ is asumed independent of $\o$.

Next, let $\M$ be a  Poisson point process of rate $1$ in $\R^d\times\R_+$,
independent of $\o$ and $\xi$. For each $\x=(x_1,\ldots,x_d)\in\Zd$, let
\begin{equation}
	\M_{\x} = \M\cap \left(C_{\x}\times \R_+\right), 
	\label{2.2}
\end{equation}
where $C_{\x}=\bigtimes\limits_{{i=1}}^{{d}}\left[c_{x_i},%
c_{x_i}+1\right)$, with $c_{x_i}:=x_i-1/2$, $1\leq i\leq d$. It is quite clear that
\begin{equation}
	\M=\bigcup\limits_{\x\in\Zd}\M_{\x}
	\label{2.3}
\end{equation}
and that by well known  properties of Poisson point processes, $\left\{\M_{\x}:\x\in\Zd\right\}$ is an independent collection such with $\M_{\x}$ 
a  Poisson point process of rate $1$ in $C_{\x}\times [0,+\infty)$. 

Given $\o\in A$ and 
$\vf:\N\to(0,1]$, 
set
\begin{equation}
	\Nx=\left\{(y_1,\ldots,y_d,r)\in\M_{\x}:y_d\in\left[\,c_{x_d},c_{x_d}+\vf(\o_{\x}(r))\,\right)\right\}, \quad \x\in\Zd.
	\label{2.4}
\end{equation}
\noindent Note that the projection of $\Nx$ on $\{\x\}\times \R_+$ is a 
inhomogeneous Poisson point process on $\{\x\}\times \R_+$ with intensity function given by
\begin{equation}
	\lambda_{\x}(r)=\vf(\o_{\x}(r)),\quad \x\in\Zd,~r\geq 0.
	\label{2.5}
\end{equation}

Let us fix $X(0)=\x_{{0}}$, $\x_{{0}}\in \Zd$, and define
$X(t)$, $t\in \R_+$, as follows. 
Let $\tau_{{0}}=0$, and set
\begin{equation}
	\tau_{1}=\inf\left\{r>0:\cN_{\x_0}\cap\left(C_{\x_0}\times \left(0,r\right]\right)\neq \emptyset\right\},
	\label{2.6}
\end{equation}
\noindent where by convention $\inf\emptyset=\infty$. For $t\in (0,\tau_{{1}})$, $X(t)=X(0)$, and, if $\tau_{1}<\infty$, then
\begin{equation}
	X\left(\tau_{{1}}\right)=X(0)+\xi_1.
	\label{2.7}
\end{equation}

\noindent For $n\geq 2$, we inductively define
\begin{equation}
	\tau_{n}=\inf\left\{r>\tau_{n\text{-}{1}}:\cN_{X_{\tau_{{n\text{-}1}}}}
	\cap\left(C_{X_{\tau_{n\text{-}{1}}}}\times \left(\tau_{n\text{-}{1}},
	r\right]\right)\neq \emptyset\right\}.
	\label{2.8}
\end{equation}
\noindent For $t\in\left(\tau_{n\text{-}{1}},\tau_{n}\right)$, we set $X(t)=
X\left(\tau_{n\text{-}{1}}\right)$, and, if $\tau_{n}<\infty$, then

\begin{equation}
	X\left(\tau_{{n}}\right)=X\left(\tau_{n-1}\right)+\xi_n.
	\label{2.9}
\end{equation}

In words, $\left(\tau_n\right)_{n\in \N}$ are the jump times of the process $X:=\left(X(t)\right)_{t\in\R_+}$, which in turn, given $\o\in A$, is a continuous time random walk on $\Zd$ starting from $\x_{{0}}$ with jump rate at $\x$ at time $t$ given by $\vf(\o_{\x}(t))$,
$\x\in\Zd$. Moreover, when at  $\x$, the next site to be visited is given by $\x+\y$, with $\y$ generated from $\pi$, 
$\x,\y\in\Zd$.
We adopt $\cD\left(\R_+,\Zd\right)$ as sample space for $X$. 

Let us denote by $P_{\x_0}^{^{\o}}$ the conditional law of $X$ given $\o\in A$. 
We remark that, since $\cN_{\x}\subset \M_{\x}$ for all $\x\in\Zd$, it follows from the lack of memory of Poisson processes that, 
for each $n\in \N_*$, given that $\tau_{n-1}<\infty$, $P_{\x_0}^{^{\o}}$-almost surely $(P_{\x_0}^{^{\o}}$-a.s.), $\tau_n-\tau_{n-1}\geq Z_n$, with
$Z_n$ a standard exponential random variable. Thus, $\tau_n\to\infty$ $P_{\x_0}^{^{\o}}$-a.s.~as $n\to\infty$, i.e., $X$ is non-explosive. Thus, given $\o\in A$, the inductive construction of $X$ proposed above  is well defined for all $t\in\R_+$.  We also notice that given the ergodicity assumption we made on $\o$, we also have that $X$ gives $P_{\x_0}^{^{\o}}$-a.s.~infinitely many jumps along all of its history for almost every realization of $\o$. 

Let us denote by $\sx=\left(\xn\right)_{n\in\N}$
the embedded (discrete time) chain of $X$. We will henceforth at times make reference to a {\em particle} which moves in continuous time on $\Zd$, 
starting from $\x_0$, and whose  trajectory is given by $X$; in this context, $X(t)$ is of course the position of the particle at time $t\geq0$.
For simplicity, we assume $\sx$ {\em irreducible}.

\begin{observacao}
	At this point it is worth pointing out that, given $\o$, $X$ is a time inhomogeneous Markov jump process; we also have that the joint process
	$\left(X(t),\o(t)\right)_{t\in\R_+}$ is Markovian.
\end{observacao}

We may then realize our joint process in the triple
$(\Omega,\F,\P_{\hat{\mu}_{{0}},\x_0})$, with $\hat{\mu}_{{0}},\x_0$ as above, where
$\Omega=\cD\left(\R_+,\N\right)^{\Zd}\times\cD\left(\R_+,\Zd\right)$,
$\F$ is the appropriate product $\sigma$-algebra on $\Omega$, and 

\begin{equation}
	\P_{\hat{\mu}_{{0}},\x_0}\left(M\times N\right)=\int_{M}
	\dif\Pm_{\hat{\mu}_{{0}}}(\omega)P_{\x_0}^{\omega}(N),
	\label{2.13}
\end{equation}
where $M$ an $N$ are measurable subsets from $A$ e $\cD\left(\R_+,\Zd
\right)$, respectively. We will call
$P_{\x_0}^{^{\o}}$ the \textit{quenched} law of $X$ (given $\o$), and  $\P_{\hat{\mu}_{{0}},\x_0}$ the \textit{annealed} law of $X$. 

We will say that a claim about $X$ holds $\P_{\x_0,\hat{\mu}_{{0}}}$-a.s.~if for $\Pm_{\hat{\mu}_{{0}}}$-almost every~$\o$
(for $\Pm_{\hat{\mu}_{{0}}}$-a.e. $\o$), the claim holds $P^{\o}_{\x_0}$-q.c. 

We will also denote by $\mathrm{E}_{\hat{\mu}_{{0}}}$, $E_{\x_0}^{^{\o}}$ and $\mathbf{E}_{\hat{\mu}_{{0}},\x_0}$ 
the expectations with respect to $\Pm_{\hat{\mu}_{{0}}}$, $P_{\x_0}^{^{\o}}$ and $\P_{\hat{\mu}_{{0}},\x_0}$, respectively. 
We reserve the notation $\Pr_\mu$ (resp., $\Pr_n$) and $\Er_\mu$  (resp., $\Er_n$) for the probability and its expectation underlying
a single birth-and-death process (as specified above) starting from a initial distribution $\mu$ on $\N$ (resp., starting from $n\in\N$).

Furthermore, in what follows, without loss, we will adopt $\x_0\equiv\zero$, and omit such a subscript, i.e.,
\begin{equation}
	P^{^{\o}}:=P_{\zero}^{^{\o}} \quad \text{and} \quad 
	\P_{\hat{\mu}_{{0}}}:=\P_{\hat{\mu}_{{0}},\zero}.
\end{equation}
\noindent We will also omit the subscript $\hat{\mu}_{{0}}$ when it is 
irrelevant. And from now on we will indicate
\begin{equation}
	\P_{\bw}, \quad \bw\in \Lambda,
	\label{eq:2.15}
\end{equation}
\noindent the law of the joint process starting from $\o(0)=\bw$ and $\x_0\equiv \zero$.

Let now $\D_n:=\tn-\t_{{n-1}}$, $n\in \N_*$. We observe that

\begin{equation}
	\P_{\hat{\mu}_{{0}}} \left(\tu>t\right)=
	\Em_{\hat{\mu}_{{0}}}\left[\exp{\left(-\int_0^t 
		\vf(\omega_{\zero}(s))\,ds\right)}\right], \quad t\in\R_+,
	\label{2.15}
\end{equation}

\noindent and, for $n\in\N$, 

\begin{equation}
	\P_{\hat{\mu}_{{0}}} \left(\D_{n+1}>t\right)=
	\Em_{\hat{\mu}_{{0}}}\left[\exp{\left(-\int_{\tn}^{\tn+t} 
		\vf(\omega_{\xn}(s))\,ds\right)}\right], \quad t\in\R_+,
	\label{2.16}
\end{equation}

\noindent recalling that $\left(\xn\right)_{n\in\N}$ denotes the jump chain of $\left(X(t)\right)_{t\in\R_+}$. For $n\in\N$, let us set

\begin{equation}
	I_n(t):=\int_{\tn}^{\tn+t} \vf(\omega_{\xn}(s))\,ds,\quad t\in \R_+,
	\label{2.18}
\end{equation}

\noindent
$I_n:\R_+\to \R_+$, $n\in\N$, is  well defined and is invertible $\P$-q.c.~under our conditions on the parameters of $\o$ (which ensure its recurrence).    
We may thus write
\begin{equation}                                                        
	\P_{\hat{\mu}_{{0}}} \left(\t_1>t\right)=            
	\Em_{\hat{\mu}_{{0}}}\left[e^{-I_{0}(t)}\right], \quad t\in\R_+,
	\label{2.19}
\end{equation}
\noindent and 
\begin{equation}
	\P_{\hat{\mu}_{{0}}} \left(\D_{n+1}>t\right)=
	\Em_{\hat{\mu}_{{0}}}\left[e^{-I_n\left(t\right)}\right], 
	\quad t\in\R_+.
	\label{2.20}
\end{equation}

\subsection{Alternative construction}
\label{sec:2.2}
\noindent We finish this section with an alternative construction of $X$, based in the following simple remark, which will be used further on.
\noindent Let $\o$ and $\xi$ as above be fixed, and set $\mathsf{T}_0=0$ and, for $n\in\N_*$, $\mathsf{T}_n=\sum\limits_{{k=0}}^{{n-1}} I_{k}\left(\D_{k+1}\right)$. 
\begin{lema}
	Under the conditions on the parameters of $\o$ assumed in the paragraph of~\eqref{erg}, we have that
	$\left\lbrace \mathsf{T}_n:n\in\N_*\right\rbrace$ is a rate 1 Poisson point process on $\R_+$, independent 
	of $\o$ and $\xi$. 
	\label{lema:2.2}
\end{lema}

\begin{proof}
	It is enough to check that, given $\o$ and $\xi$, $\left(\D_{n}\right)_{n\in\N_*}$ are the event times of a
	Poisson point process, which are thus independent of each other; the conclusion follows readily from the fact that
	\begin{equation*}
		\P\big(I_n(\D_{n+1})>t\big)=
		\P\big(\D_{n+1}>I_n^{-1}(t)\big)=
		\Em\left[e^{-I_n\left(I_n^{-1}(t)\right)}\right]
		=e^{-t},~ t\in\R_+. 
		\label{2.26}
	\end{equation*}
\end{proof}

We thus have an alternative construction of $X$, as follows.
Let $\o$, $\xi$ be as described at the beginning of the section. Let also $\sV=\left(\sV_n\right)_{n\in\N}$ 
be an indepent family of standard exponential random variables. Then, given $\o$, set $X(0)=\xz\equiv \zero$ and $\tau_{{0}}=0$,
and define
\begin{equation}
	\tau_{1}=I^{-1}_0(\sV_1).
	\label{2.30}
\end{equation}
\noindent For all $t\in (0,\tau_{{1}})$, $X(t)=X(0)$ and
\begin{equation}
	X(\tau_{{1}})=X(0)+\xi_1=\xu.
	\label{2.31}
\end{equation}
set, inductively,
\begin{equation}
	\tau_{n}=\tau_{n-1}+I_{n-1}^{-1}(\sV_n),
	\label{2.32}
\end{equation}
\noindent and for $t\in\left(\tau_{n{-}{1}},\tau_{n}\right)$, $X(t)=
X({\tau_{n{-}{1}}})$ and
\begin{equation}
	X(\tau_{{n}})=X(\tau_{n-1})+\xi_n=\xn.
	\label{2.33}
\end{equation} 

\noindent 

We have thus completed the alternative construction of $X$. Notice that we have made use of $\o$ and $\xi$, as in the original construction,
but replaced $\M$ of the latter construction by  $\sV$ as the remaining ingredient.
The alternative construction comes in handy in a coupling argument we develop in order to prove a law of large numbers for the jump times of $X$.



\section{Limit theorems under  $\Pz$}
\label{conv}

\setcounter{equation}{0}

\noindent 
In this section state and prove two of our main results, namely a Law of Large Numbers and a Central Limit Theorem for $X$ under  $\Pz$\footnote{We recall that $\P_{\bw}$ represents the law of $X$ starting from $\o(0)=\bw$ and $\x_0=\zero$.} and under the following extra conditions on $\vf$:
\begin{equation}\label{vf}
	\vf \mbox { is non increasing},\, \vf(0)=1, \mbox{ and } \lim_{n\to\infty}\vf(n)=0.
\end{equation}
The statements are provided shortly, and the proofs are presented in the second and third subsections below, respectively. The main ingredient for these results is a Law of Large Numbers for the jump time of $X$, which in turn uses a stochastic domination result for the distribution of the environment seen by the particle at jump times; both results, along with other preliminay material, are developed  in the first subsection below.

In order to state the main results of this section, we need the following preliminaries and further conditions on $\p,\q$. Let $\nu$ denote the invariant distribution of $\o_\zero$, such that, as is well known, $\nu_n=$ const $\prod_{i=1}^n\frac{p_{i-1}}{q_i}$, for $n\in\N$, where the latter product is conventioned to equal 1 for $n=0$. Next set $\rho_n=\frac{p_n}{q_n}$, $R_n=\prod_{i=1}^n\rho_i$ and $S_n=\sum_{i\geq n} R_i$ , $n\geq1$,
and let $R_0=1$. These quantities are well defined and, in particular, it follows from~\eqref{erg} that the latter sum is finite for all $n\geq1$.

We will require the following extra condition on $\p,\q$, in addition to those imposed in the paragraph of~\eqref{erg} above: we will assume
\begin{equation}\label{extracon}
	\sum_{n\geq1}\frac{S_n^2}{R_n}<\infty.
\end{equation}
We note that it follows from our previous assumptions on $\p,\q$ that~\eqref{extracon} is stronger than~\eqref{erg}, since $S_n\geq R_n$ for all $n$.
The relevance of this condition is that it implies 
the two conditions to be introduced next. 

Let $\w$ denote the embedded chain of $\omega_\zero$, and, for $n\geq0$, let $T_n$ denote the first passage time of $\w$ by $n$, namely, $T_n=\inf\{i\geq0:\,\w_i=n\}$, with the usual convention that $\inf\emptyset=\infty$.
Condition~\eqref{extracon} is equivalent, as will be argued in Appendix~\ref{app}, to either
\begin{equation}\label{extracor}
	\Er_\nu(T_0)<\infty \mbox{ or } \Er_1(T_0^2)<\infty. \footnotemark
\end{equation}
\footnotetext{We note that $\w$ has the same invariant distribution as $\o_\zero$, namely $\nu$.}
It may be readily shown to be stronger than asking that $\nu$ have a finite first moment, and a finite second moment of $\nu$ implies it, 
under our conditions on $\p,\q$ \footnote{It looks as though a finite second moment of $\nu$ may be a necessary condition for it, as well.}. 
Conditions~\eqref{extracor}  will be required in our arguments for the following main results of this section --- 
they are what we meant by 'strongly ergodic' in the abstract. 
See Remark~\ref{relax} at the end of this section. 

\begin{teorema}[Law of Large Numbers for $X$] 
	
	Assume the above conditions and that $\E(\|\xi_1\|)<\infty$. Then there exists $\mu\in(0,\infty)$ such that
	\begin{equation}
		\frac{X(t)}{t}\to \frac{\E(\xi_1)}{\mu} ~~\P_{\zero}\mbox{-a.s.}~\text{as}~t\to\infty.
		\label{eq:63}
	\end{equation}
	\label{teo:3.1}
\end{teorema}

Here and below $\|\cdot\|$ is the sup norm in $\Zd$.

\begin{teorema} [Central Limit Theorem for $X$] 
	
	Assume the above conditions  and that $\E(\|\xi_1\|^2)<\infty$ and $\E(\xi_1)=\zero$. Then, for $\Pm_{\zero}$-a.e. $\o$, we have that
	\begin{equation}
		\frac{X(t)}{\sqrt{ t/\mu}}\Rightarrow N_d(\zero,\Sigma) \,\mbox{ under }\, P^{^{\o}}, 
		\label{eq:75}
	\end{equation}
	\label{teo:3.2}
	\noindent where $\Sigma$ is the covariance matrix of $\xi_1$, and $\mu$ is as in Theorem~\ref{teo:3.1}.
\end{teorema}

In the next section we will state a CLT under more general initial environment conditions (but restricting to homogeneous cases of the environmental BD dynamics).
As for the mean zero assumption in Theorem~\ref{teo:3.2}, going beyond it would require substantially more work than we present here, under our approach; see Remark~\ref{ext_clt} at the end of this section.

\begin{subsection}{Law of large numbers for the jump times of $X$}
	\label{sec:3.1}
	
	\noindent In this subsection, we prove a Law of Large Numbers for $(\tn)_{n\in\N}$ under $\P_\zero$; this is the key ingredient in our arguments for the main results of this section; see Proposition~\ref{prop:3.1} below. 
	Our strategy for proving the latter result is to establish suitable stochastic domination of the environment by a modified environment, leading to a corresponding domination for jump times; we develop this program next.

	\medskip

	We start by recalling some well known definitions. Given two probabilities on $\N$, $\upsilon_1$ and $\upsilon_2$, 
	we indicate by $\upsilon_1\preceq \upsilon_2$ that $\upsilon_1$ is stochastically dominated by $\upsilon_2$, i.e., 
	\begin{equation}
		\upsilon_1\left(\N~\setminus~ \mathsf{A}_k \right)\leq
		\upsilon_2\left(\N~\setminus~ \mathsf{A}_k\right),
		\quad \mathsf{A}_k:=\left\lbrace 0,\ldots, k\right\rbrace, ~ \forall~k\in\N.
	\end{equation} 
	We equivalently write, in this situation, $X_1\preceq \upsilon_2$, if $X_1$ is a random variable distributed as $\upsilon_1$.

	Now let $\sQ$ denote the generator of $\o_\zero$ (which is a {\em Q-matrix}), and consider the following matrix
	\begin{equation}
		\sQ^\psi=D\sQ,~\textrm{with}~D=\mathrm{diag}\{\psi(n)\}_{n\in\N},
		\label{eq:3.3}
	\end{equation}
	where $\psi:\N\to[1,\infty)$ is such that $\psi(n)=1/\vf(n)$ for all $n$, with $\vf$ as defined in the paragraph of~\eqref{2.4} above.
	\noindent 
	Notice that $\sQ^\psi$ is also a Q-matrix, and that it generates a birth-and-death process on $\N$, say $\cho_\zero$, with transition rates given by
	\begin{equation}
		\sQ^\psi(n,n+1)=\psi_np_n=:p^\psi_n,\,n\in\N\, ; \quad \sQ^\psi(n,n-1)=\psi_nq_n=:\qp_n,\, n\in\N_*; \footnotemark
		\label{3.16}
	\end{equation}
	\footnotetext{We occasionally use $g_n$ to denote $g(n)$, for $g:\N\to\R$.}
	this is a positive recurrent process, with invariant distribution $\nu^\psi$ on $\N$ such that 
	\begin{equation*}
		\nu^\psi_ n=\mbox{const }\prod_{i=1}^n\frac{\pp_{i-1}}{\qp_i},\,n\in\N,
	\end{equation*}
	with a similar convention for the product as for $\nu$. 
	One may readily check  that $\nu_\psi\preceq\nu$, since $\psi$ is increasing. 
	The relevance of $\cho_\zero$ in the present study issues from the following strightforward result. Recall~\eqref{2.18}.
	
	\begin{lema}
		Suppose $\omega_\zero(0)\sim\cho_\zero(0)$. Then
		\begin{equation}\label{coup1}
			(\omega_\zero(t),\,t\in\R_+)\sim(\cho_\zero(I_0(t)),\,t\in\R_+).
		\end{equation}
		\label{coup}
	\end{lema}
	We have the following immediate consequence from this and Lemma~\ref{lema:2.2}.
	\begin{corolario}
		Let $\sV_1$ be a standard exponential random variable, independent of $\cho_\zero$. Then
		\begin{equation}
			\omega_\zero(\tau_1)\sim\cho_\zero(\sV_1).
			\label{coup3}
		\end{equation}
		\label{coup2}
	\end{corolario}
	Figure~\ref{fig:3.1} illustrates a coupling behind~(\ref{coup1},\ref{coup3}).
	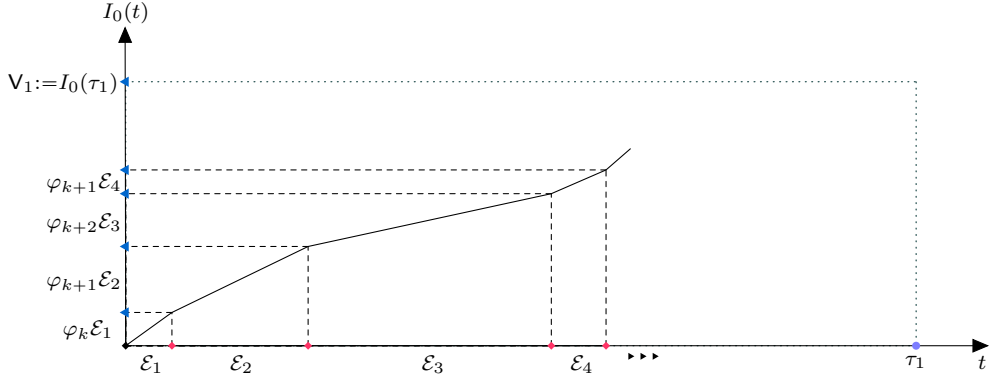
\begin{figure}[!htb]
		\centering
         \definecolor{qqtttt}{rgb}{0,0.2,0.2}
         \definecolor{xdxdff}{rgb}{0.49,0.49,1}
         \definecolor{qqwwcc}{rgb}{0,0.4,0.8}
         \definecolor{ffttww}{rgb}{1,0.2,0.4}
         \begin{tikzpicture}[line cap=round,line join=round,>=triangle 45,x=0.8cm,y=0.7cm]
         	\clip(-2.10,-2.27) rectangle (15.25,6.93);
         	\draw (0.76,0)-- (3,0);
         	\draw (3,0)-- (7,0);
         	\draw (7,0)-- (7.9,0);
         	\draw (0.76,0.63)-- (3,1.88);
         	\draw (3,1.88)-- (7,2.88);
         	\draw (7,2.88)-- (7.9,3.33);
         	\draw [dash pattern=on 2pt off 2pt] (0,0)-- (0.76,0);
         	\draw [dash pattern=on 2pt off 2pt] (0.76,0)-- (0.76,0.63);
         	\draw [dash pattern=on 2pt off 2pt] (0.76,0.63)-- (0,0.63);
         	\draw [dash pattern=on 2pt off 2pt] (0,0.63)-- (0,0);
         	\draw [dash pattern=on 2pt off 2pt] (0,0)-- (3,0);
         	\draw [dash pattern=on 2pt off 2pt] (3,0)-- (3,1.88);
         	\draw [dash pattern=on 2pt off 2pt] (3,1.88)-- (0,1.88);
         	\draw [dash pattern=on 2pt off 2pt] (0,1.88)-- (0,0);
         	\draw [dash pattern=on 2pt off 2pt] (0,0)-- (7,0);
         	\draw [dash pattern=on 2pt off 2pt] (7,0)-- (7,2.88);
         	\draw [dash pattern=on 2pt off 2pt] (7,2.88)-- (0,2.88);
         	\draw [dash pattern=on 2pt off 2pt] (0,2.88)-- (0,0);
         	\draw [dash pattern=on 2pt off 2pt] (0,0)-- (7.9,0);
         	\draw [dash pattern=on 2pt off 2pt] (7.9,0)-- (7.9,3.33);
         	\draw [dash pattern=on 2pt off 2pt] (7.9,3.33)-- (0,3.33);
         	\draw [dash pattern=on 2pt off 2pt] (0,3.33)-- (0,0);
         	\draw [dotted,color=qqtttt] (13,0)-- (0,0);
         	\draw [dotted,color=qqtttt] (0,0)-- (0,5);
         	\draw [dotted,color=qqtttt] (0,5)-- (13,5);
         	\draw [dotted,color=qqtttt] (13,5)-- (13,0);
         	\draw (-2.1,5.35) node[anchor=north west] {$\ms{\sV_1:=I_0\left(\tau_1\right)}$};
         	\draw (12.65,0.0) node[anchor=north west] {$\ms{\tau_1}$};
         	\draw (13.85,0.0) node[anchor=north west] {$\ms{t}$};
         	\draw (0.07,0.0) node[anchor=north west] {$\ms{\cE_1}$};
         	\draw (1.55,0.0) node[anchor=north west] {$\ms{\cE_2}$};
         	\draw (4.7,0.0) node[anchor=north west] {$\ms{\cE_3}$};
         	\draw (7.15,0.0) node[anchor=north west] {$\ms{\cE_4}$};
         	\draw (-1.2,0.67) node[anchor=north west] {$\ms{\vf_k\cE_1}$};
         	\draw (-1.47,1.60) node[anchor=north west] {$\ms{\vf_{k+1}\cE_2}$};
         	\draw (-1.47,2.69) node[anchor=north west] {$\ms{\vf_{k+2}\cE_3}$};
         	\draw (-1.47,3.45) node[anchor=north west] {$\ms{\vf_{k+1}\cE_4}$};
         	\draw [->] (0,0) -- (14.2,0);
         	\draw (-0.55,6.7) node[anchor=north west] {$\ms{I_0(t)}$};
         	\draw (7.9,3.33)-- (8.3,3.73);
         	\draw (0,0)-- (0.76,0.63);
         	\draw [->] (0,0) -- (-0.01,6.05);
         	
         	\begin{scriptsize}
         		\fill [color=black] (0,0) ++(-1.5pt,0 pt) -- ++(1.5pt,1.5pt)--++(1.5pt,-1.5pt)--++(-1.5pt,-1.5pt)--++(-1.5pt,1.5pt);
         		\fill [color=ffttww] (0.76,0) ++(-1.5pt,0 pt) -- ++(1.5pt,1.5pt)--++(1.5pt,-1.5pt)--++(-1.5pt,-1.5pt)--++(-1.5pt,1.5pt);
         		\fill [color=ffttww] (3,0) ++(-1.5pt,0 pt) -- ++(1.5pt,1.5pt)--++(1.5pt,-1.5pt)--++(-1.5pt,-1.5pt)--++(-1.5pt,1.5pt);
         		\fill [color=ffttww] (7,0) ++(-1.5pt,0 pt) -- ++(1.5pt,1.5pt)--++(1.5pt,-1.5pt)--++(-1.5pt,-1.5pt)--++(-1.5pt,1.5pt);
         		\fill [color=ffttww] (7.9,0) ++(-1.5pt,0 pt) -- ++(1.5pt,1.5pt)--++(1.5pt,-1.5pt)--++(-1.5pt,-1.5pt)--++(-1.5pt,1.5pt);
         		\fill [color=qqwwcc,shift={(0,0.63)},rotate=90] (0,0) ++(0 pt,2.25pt) -- ++(1.95pt,-3.375pt)--++(-3.9pt,0 pt) -- ++(1.95pt,3.375pt);
         		\fill [color=qqwwcc,shift={(0,1.88)},rotate=90] (0,0) ++(0 pt,2.25pt) -- ++(1.95pt,-3.375pt)--++(-3.9pt,0 pt) -- ++(1.95pt,3.375pt);
         		\fill [color=qqwwcc,shift={(0,2.88)},rotate=90] (0,0) ++(0 pt,2.25pt) -- ++(1.95pt,-3.375pt)--++(-3.9pt,0 pt) -- ++(1.95pt,3.375pt);
         		\fill [color=qqwwcc,shift={(0,3.33)},rotate=90] (0,0) ++(0 pt,2.25pt) -- ++(1.95pt,-3.375pt)--++(-3.9pt,0 pt) -- ++(1.95pt,3.375pt);
         		\fill [color=qqwwcc,shift={(0,5)},rotate=90] (0,0) ++(0 pt,2.25pt) -- ++(1.95pt,-3.375pt)--++(-3.9pt,0 pt) -- ++(1.95pt,3.375pt);
         		\fill [color=xdxdff] (13,0) circle (1.5pt);
         		\fill [color=black,shift={(8.3,-0.2)},rotate=270] (0,0) ++(0 pt,1.5pt) -- ++(1.3pt,-2.25pt)--++(-2.6pt,0 pt) -- ++(1.3pt,2.25pt);
         		\fill [color=black,shift={(8.5,-0.2)},rotate=270] (0,0) ++(0 pt,1.5pt) -- ++(1.3pt,-2.25pt)--++(-2.6pt,0 pt) -- ++(1.3pt,2.25pt);
         		\fill [color=black,shift={(8.7,-0.2)},rotate=270] (0,0) ++(0 pt,1.5pt) -- ++(1.3pt,-2.25pt)--++(-2.6pt,0 pt) -- ++(1.3pt,2.25pt);
         	\end{scriptsize}
         	
         \end{tikzpicture}
		\caption{$\cE_1,\cE_2,\ldots$ are iid standard exponentials; $x$(resp., $y$)-axis indicates constancy intervals of $\o_\zero$ (resp., $\cho_\zero$) in a realization where with $\o_{\zero}(0)=\cho_{\zero}(0)=  k\in\N$.}
		\label{fig:3.1} 
	\end{figure}

	The following result is most certainly well known, and may be argued by a straightforward coupling argument.
	
	\begin{lema}
		Let $\mu$ and $\mu'$ denote two probabilities on $\N$ such that $\mu\preceq \mu'$. Then, 
		for all $t\in\R_+$,
		\begin{equation}
			\mu e^{t\sQ}\preceq \mu' e^{t\sQ}. 
			\label{3.1}
		\end{equation}
		\label{lema:3.1}
	\end{lema}
	
	Here and below $e^{t\sQ'}$ denotes the semigroup associated to an irreducible and recurrent Q-matrix $\sQ'$ on $\N$. 
	We have an immediate consequence of Lemma~\ref{lema:3.1}, as follows.
	
	\begin{corolario}
		If $\mu$ is a probability on $\N$ such that $\mu\preceq \nu$, then, for all $t\in\R_+$,
		\begin{equation}
			\mu e^{t\sQ}\preceq \nu.
			\label{3.0}
		\end{equation}
		\label{coro:3.0}
	\end{corolario}
	We present now a few  more substantial domination lemmas, 
	leading to a key ingredient for justifying 
	the main result of this subsection.
	
	\begin{lema}\label{domi1}
		Let $\sQ^\psi$ be as in (\ref{eq:3.3},\ref{3.16}). Then, for all $t\in\R_+$,
		\begin{equation}
			\nu e^{t\sQ^\psi}\preceq \nu. 
			\label{3.17}
		\end{equation}
		\label{lema:3.2}
	\end{lema}
	\begin{proof}
		Let $\Y=\left(\Y_t\right)_{t\in\R_+}$ denote the birth-and-death process generated by
		$\sQ^\psi$ started from $\nu$. Set $P_{n,j}(t):=\Pr(\Y_t=j~\mathbin{\vert{}}~ \Y_0=n)$, $t\in\R_+$, $n,j\in\N$. 
		For $l\in\N$,
		\begin{equation}
			\Pr(\Y_t\leq l)=\sum_{j\leq l}\Pr(\Y_t=j)=
			\sum_{j\leq l}\sum_{n\geq 0}\nu_n P_{n,j}(t).
			\label{3.18}
		\end{equation}
		
		\noindent By Tonelli,
		\begin{equation}
			\Pr(\Y_t\leq l)=\sum_{n\geq 0} \sum_{j\leq l} \nu_n P_{n,j}(t).
			\label{3.19}
		\end{equation}

		Consider now Kolmogorov's forward equations for $\Y$, given by
		\begin{align}
			P'_{n,0}(t)&=-\pp_0P_{n,0}(t)+\qp_1P_{n,1}(t);\\
			P'_{n,j}(t)&=\pp_{j-1}P_{n,j-1}(t)-\psi_jP_{n,j}(t)+\qp_{j+1}P_{n,j+1}(t),~~j\geq 1;
			\label{3.20-3.21}
		\end{align}
		$n\geq 0$. It follows that
		\begin{equation}
			\left\lvert \sum_{{j\leq l}} \nu_nP'_{n,j}(t)\right\rvert=
			\nu_n\Big| \qp_{l+1}P_{n,l+1}(t)-\pp_{l}P_{n,l}(t)\Big|\leq \nu_n\psi_{l+1},
			\label{3.22}
		\end{equation}
		for all $t$; since $\nu$ is summable, we have that
		\begin{equation}
			\Pr'(\Y_t\leq l)=\sum_{n\geq 0}  \sum_{j\leq l}\nu_nP'_{n,j}(t).
			\label{3.23}
		\end{equation}
		
		We now make use  in \eqref{3.23}  of Kolmogorov's backward equations for $\Y$, given by
		\begin{align}
			P'_{0,j}(t)&=-\pp_0P_{0,j}(t)+\pp_0P_{1,j}(t)=\pp_0(P_{1,j}(t)-P_{0,j}(t));\\
			P'_{n,j}(t)&=\qp_{n}P_{n-1,j}(t)-\psi_nP_{n,j}(t)+\pp_{n}P_{n+1,j}(t),\nonumber\\ 
			&=\qp_{n}(P_{n-1,j}(t)-P_{n,j}(t))-\pp_{n}(P_{n,j}(t)-P_{n+1,j}(t)),\quad n\geq 1,
			\label{3.24,3.25}
		\end{align}
		$j\geq 0$. Setting $d_n:=\Pr_n(\Y_t\leq l)-\Pr_{n+1}(\Y_t\leq l)$, $n\in\mathbb N$,  
		we find that
		\begin{align}
			\Pr'(\Y_t\leq l)&=\sum_{j\leq l}\nu_0P'_{0,j}(t)+
			\sum_{n\geq 1}\sum_{j\leq l}\nu_nP'_{n,j}(t)\nn\\
			&=-\nu_0\pp_0d_0+\sum_{n {\geq} 1}\nu_n\big(\qp_nd_{n\text{-}{1}}-\pp_nd_n\big)\nn\\
			&=\sum_{n {\geq} 0}\nu_{n+1}\qp_{n+1}d_{n}-\sum_{n {\geq} 0}\nu_n\pp_nd_n,
			\label{3.26}
		\end{align}
		provided 
		\begin{equation}
			\sum_{n {\geq} 1}\nu_n\psi_n(d_{n\text{-}{1}}\vee d_{n})<\infty,
			\label{3.27}
		\end{equation}
		which we claim to hold; see justification below. We note that $d_n\geq0$ for all $n,l$ and $t$, as can be justified by a straightforward coupling argument.
		It follows that 
		\begin{align}
			\Pr'(\Y_t\leq l)&=\sum_{n {\geq} 0}(\nu_{n+1}\qp_{n+1}-\nu_{n}\pp_{n})d_{n}\nn\\
			&=\sum_{n {\geq} 0}(\psi_{n+1}\nu_{n+1}q_{n+1}-\psi_{n}\nu_{n}p_n)d_{n}\nn\\
			&=\sum_{n {\geq} 0}(\psi_{n+1}\nu_{n}p_{n}-\psi_{n}\nu_{n}p_n)d_{n}\nn\\
			&=\sum_{n {\geq} 0}(\psi_{n+1}-\psi_{n})\nu_{n}p_nd_{n}\geq0
			\label{3.27a}
		\end{align}
		since $\psi$ is nondecreasing, where the third equality follows by reversibility of $Y$. 
		
		We thus have that $\Pr(\Y_t\leq l)$ is nondecreasing in $t$ for every $l$; we thus have that
		\begin{equation}
			\nu(\mathsf{A}_l)=\Pr(\Y_0\leq l)\leq\Pr(\Y_t\leq l)
			\label{3.27b}
		\end{equation}
		for all $l$, and~\eqref{3.17} is established.
		
		\smallskip
		
		It remains to argue~\eqref{3.27}. Let $\H_n:=\inf\{s\geq 0: \Y_s=n\}$, $n\in\mathbb N$ be the hitting time of $n$ by $\Y$. For $n\geq l$,
		we have that
		\begin{align}
			d_n=&\Pr_n(\Y_t\leq l)-\int_0^t\Pr_{n+1}(\H_n\in \ds)
			\Pr_{n}(\Y_{t-s}\leq l)\ds\nn\\
			=&\int_0^t\Pr_{n+1}(\H_n\in \ds)\Big[\Pr_n(\Y_t\leq l)-
			\Pr_{n}(\Y_{t-s}\leq l)\Big]\ds\nn\\&+ 
			\Pr_n(\Y_t\leq l)\int_t^\infty \Pr_{n+1}(\H_n\in \ds)\ds\nn\\
			=&\int_0^t\Pr_{n+1}(\H_n\in \ds)\Big[\Pr_n(\Y_t\leq l,\Y_{t-s}> l)-
			\Pr_n(\Y_t> l,\Y_{t-s}\leq l)\Big]\ds\nn\\&+ \Pr_n(\Y_t\leq l)\Pr_{n+1}(\H_n>t)\nn\\
			=&:d_n'+d_n''
			\label{3.28}
		\end{align}
		
		Let now $V=\left(V_i\right)_{i\in \N_*}$ be a sequence of independent standard exponential random variables, 
		and consider the embedded chain 
		$\tilde{\Y}=\big(\tilde{\Y}_k\big)_{k\geq 0}$ of $\left(\Y_t\right)_{t\in\R_+}$, and $\tilde{\H}_n=\inf\{k\geq 0: \tilde{\Y}_k=n\}$. 
		Notice that $\tilde{\Y}$ is distributed as $\w$, and $\tilde{\H}_n$ is distributed as $T_n$, introduced at the beginning of the section.
		Let $V$ and $\tilde{\Y}$ be independent.
		Let us now introduce an auxiliary random vaiable $\H'_n=\sum\limits_{i=1}^{\tilde{\H}_n}V_i$, and note that, given that $\Y_0=n+1$,
		$\H_n\stackrel{{st}}{\preceq} \vf_{n+1}\H'_n$; it follows from this and the Markov inequality that
		\begin{equation}
			\Pr_{n+1}(\H_n>t)\leq \Pr_{n+1}(\H'_n>\psi_{n+1}t)\leq \frac{\vf_{n+1}\T_{n+1}}t\leq\mbox{ const }\vf_{n+1}\frac{S_n}{R_n}
			\label{3.29}
		\end{equation}
		(see Appendix~\ref{app}). It follows that 
		\begin{equation}
			\sum_{n {>} l}\nu_n\psi_nd''_{n-1}\leq\mbox{ const }\sum_{n {\geq} 1}\frac{\nu_n}{R_n}S_n\leq\mbox{ const }\sum_{n {\geq} 1}S_n<\infty
			\label{3.27c}
		\end{equation}
		by the ergodicity assumption on $\o_\zero$, and similarly $\sum_{n {\geq} 1}\nu_n\psi_nd''_{n}<\infty$.

		Now, by the Markov property
		\begin{align}
			\Pr_n(\Y_t\leq l,\Y_{t-s}> l)
			&=\sum_{j\geq l+1}\Pr_n(\Y_{t-s}=j)\Pr_j(\Y_{s}\leq l)\nn\\
			&\leq \sum_{j\geq l+1}\Pr_n(\Y_{t-s}=j)\Pr_{l+1}(\Y_{s}\leq l)\nn\\
			&\leq \sum_{j\geq l+1}\Pr_n(\Y_{t-s}=j)\left( 1-e^{-\psi_{l+1}s}\right)
			\leq 1-e^{-\psi_{l+1}s}.
			\label{3.31}
		\end{align}

		\noindent Thus,
		\begin{align}
			d'_n&\leq \int_0^t\Pr_{n+1}(\H_n\in \ds)(1-e^{-\psi_{l+1}s})\ds \leq 
			\Er_{n+1}\left(1-e^{-\psi_{l+1}\H_n}\right)\nn\\
			&\leq \psi_{l+1}\Er_{n+1}\left(H_n\right) \leq \psi_{l+1}\vf_{n+1}\T_{n+1},
			\label{3.32}
		\end{align}
		and, similarly as above, we find that $\sum_{n {\geq} 1}\nu_n\psi_n (d'_{n-1}+ d'_{n})<\infty$, and~\eqref{3.27} is established.
		
	\end{proof}

	In other words, if $\o_\zero(0)\sim\nu$ , then
	\begin{equation}\label{dom1}
		\o_\zero(\tau_1)\preceq\nu.
	\end{equation}
	
	Let us now assume that $\hat\mu_{\x,0}\preceq\nu$ for every $\x\in\Zd$.
	Based on the above domination results, we next construct a modification of  the joint process $(X,\o)$, to be denoted $(\brex,\bro)$, 
	in a coupled way to $(X,\o)$, so that $\bro$ has {\em less} spatial dependence than, and at the same time dominates $\o$ in a suitable way.
	The idea is to let  $\brex$ have the same embedded chain as $X$, and jump according to
	$\bro$ as $X$ jumps according to $\o$; we let $\bro$ evolve with the same law as $\o$ between its jump times, and at jump times
	we replace $\bro$ at the site where $\brex$ jumped from by a suitable dominating random variable distributed as $\nu$.
	Details follow.
	
	We first construct a sequence of environments between jumps of $\brex$, as follows.
	Let $(X,\o)$ be as above, starting from $X(0)=0$, $\o(0)\sim\hat\mu_0$, then, enlarging the original probability space if necessary,
	we can find iid random variables $\o^0_\x(0)$, $\x\in\Zd$,  distributed according to $\nu$, such that $\o^0_\x(0)\geq\o_\x(0)$, $\x\in\Zd$.
	
	We let now $\o^0$ evolve for $t\geq0$ in a coupled way with $\o$ in such a way that $\o^0_\x(t)\geq\o_\x(t)$, $\x\in\Zd$. 
	Let now $\bret_1$ be obtained from $\o^0$ in the same way as $\tau_1$ was obtained from $\o$, using the same $\M$ for $\o^1$ as for $\o$ 
	(recall definition from paragraph of~\eqref{2.2}); $\bret_{1}$ is the time of the first jump of $\brex$, and set $\brex(\bret_1)=\sx_1$. 
	Notice that $\bret_1\geq\tau_1$.
	
	Noticing as well that $\o^0_\x(\bret_1)$, $\x\ne0$, are independent  with common distribution $\nu$, and independent of $\o^0_\zero(\bret_1)$,
	and using~\eqref{dom1}, again enlarging the probability space if necessary, we find $\W_1$ distributed as $\nu$ such that 
	$\W_1\geq\o^0_\zero(\bret_1)$, with $\W_1$ is independent of $\o^0_\x(\bret_1)$, $\x\ne0$;
	and we make $\o^1_\zero(\bret_1)=\W_1$, and $\o^1_\x(\bret_1)=\o^0_\x(\bret_1)$, $\x\ne0$.
	Notice that $\o^1_\x(\bret_1)$, $\x\in\Zd$ are iid with marginals distributed as $\nu$.

	We now iterate this construction, inductively: given $\xi$, let us fix $n\geq1$, and suppose that for each $0\leq j\leq n-1$, we have  constructed $\bret_j$, and $\o^{j}(t),\,t\geq\bret_{j}$, with $\{\o^j_\x(\bret_j), \,\x\in\Zd\}$ iid with marginals distributed as $\nu$. We then define $\bret_n$ from $\o^{n-1}(\bret_{n-1})$ in the same way as $\tau_1$ was defined from $\o^0(0)$, but with the random walk originating in $\sx_{n-1}$, and with the
	marks of $\M$ in the upper half space from $\bret_{n-1}$; $\bret_{n}$ is the time of the $n$-th jump of $\brex$, and we set $\brex(\bret_{n})=\sx_{n}$. 
	
	Next, from~\eqref{dom1}, we obtain $\W_n\geq\o^{n-1}_{\x_{n-1}}(\bret_{n})$ such that $\{\W_n; \o^{n-1}_{\x}(\bret_{n}),\,\x\ne\sx_{n-1}\}$ is an iid family of random variables with marginals distributed as $\nu$, and define a $BDP(\p,\q)$ $(\o^n(t))_{t\geq\bret_n}$ starting from 
	$\{\o^n_\x(\bret_n)=\o^{n-1}_\x(\bret_n),\,\x\ne\sx_{n-1}; \,\o^n_{\x_{n-1}}(\bret_n)=\W_n\}$ so that 
	$\o^n_{\sx_{n-1}}(t)\geq\o^{n-1}_{\sx_{n-1}}(t)$, 
	$\o^n_\x(t)=\o^{n-1}_\x(t)$, $\x\ne\sx_{n-1}$, $t\geq\bret_n$.

	We finally define $\bro(t)=\o^n(t)$ for $t\in[\bret_{n},\bret_{n+1})$, $n\geq0$. This coupled construction of $(\o,\bro)$ has the following properties.
	
	
	\begin{lema}\label{domult}
		\mbox{}
		\begin{enumerate}
			\item \begin{equation}\label{comp1}
				\bro_\x(t)\geq\o_\x(t) \text{ for all }\,\x\in\Zd \text{ and }\, t\geq0;
			\end{equation}
			\item for each $n\geq0$,
			\begin{equation}\label{comp2}
				\bro_\x(\bret_n), \, \x\in\Zd,\text{ are iid random variables with marginals distributed as }\nu;
			\end{equation}
			\item for all $n\geq0$, we have that
			\begin{equation}\label{comp3}
				\tau_n\leq\bret_n.
			\end{equation}
		\end{enumerate}	
	\end{lema}
	\begin{proof}
		The first two items are quite clear from the construction, so we will  argue only the third item, which is quite clear for $n=0$ and $1$
		(the latter  case was already pointed out  in the description of the construction, above); for the remaining cases, let $n\geq1$, 
		and suppose, inductively, that $\tau_n\leq\bret_n$; there are two
		possibilities for $\tau_{n+1}$: either $\tau_{n+1}\leq\bret_n$, in which case, clearly, $\tau_{n+1}\leq\bret_{n+1}$, or $\tau_{n+1}>\bret_n$; 
		in this latter case, $\tau_{n+1}$ (resp., $\bret_{n+1}$)
		will correspond to the earliest Poisson point (of $\M$) in $\cQ_n:=[c_{\sx_n(d)},c_{\sx_n(d)}+\vf(\o_{\sx_n}(r))_{r\geq\bret_n}$ 
		(resp., $\breve\cQ_n:=[c_{\sx_n(d)},c_{\sx_n(d)}+\vf(\bro_{\sx_n}(r))_{r\geq\bret_n}$). By~\eqref{comp1} and the monotonicity of $\vf$, we have that
		$\breve\cQ_n\subset\cQ_n$, and it follows that $\tau_{n+1}\leq\bret_{n+1}$.
	\end{proof}
	The next result follows immediately.
	\begin{corolario}\label{domesp}
		For $n\geq1$
		\begin{equation}\label{dom2}
			\E_{\hat\mu_0}\left(\tau_n\right)\leq	\E_{\hat\nu}\left(\bret_n\right)=n\E_{\hat\nu}\left(\bret_1\right)=n\E_{\hat\nu}\left(\tau_1\right).
		\end{equation}
	\end{corolario}

	The following result, together with~\eqref{dom2}, is a key ingredient in the justification of the main result of this subsection.
	
	\begin{lema}\label{fin}
		\begin{equation}\label{fin1}
			\E_{\hat{\nu}}\left(\tu\right)<\infty
		\end{equation}
	\end{lema}
	
	\begin{proof}
		Let us write
		\begin{align}
			\E_{\hat{\nu}}\left(\tau_1\right)&=\int_{0}^{\infty}\P_{\hat{\nu}}\left(\tau_1>t
			\right)dt\nn\\
			&=\int_{0}^{\infty}\Em_{\hat{\nu}}\left(e^{-I_{\zero}(t)}\right)dt\nn\\
			&=\int_{0}^{+\infty}\Em_{\hat{\nu}}\left(e^{-I_0(t)};I_0(t)\geq \epsilon t\right)dt
			+\int_{0}^{+\infty}\Em_{\hat{\nu}}\left(e^{-I_0(t)};I_0(t)<\epsilon t\right)dt\nn\\
			&\leq \epsilon^{-1}+\int_{0}^{+\infty}\P_{\hat{\nu}}\left(I_0(t)<\epsilon t\right)dt\nn \\
			&\leq \epsilon^{-1}+\int_{0}^{+\infty}\P_{\hat{\nu}}\left(\int_{0}^t 
			\mathds{1}\left\lbrace\o_{\mzero}(s)=0\right\rbrace ds<\epsilon t\right)dt.
			\label{eq:51}
		\end{align}
		
		For $k\in\N$, set $\k=k\times \one$. Conditioning in the initial state of the environment at the origin, 
		we have, for each  $\d>0$ and each $t\in\R_+$,
		\begin{eqnarray}
			\P_{\hat{\nu}}\left(\int_{0}^t\mathds{1}\left\lbrace \o_{\mzero}(s)=0\right\rbrace ds<\epsilon t\right)&=&
			\sum_{k=1}^{\left\lfloor \delta t\right\rfloor}\nu_k\,\P_{\k}
			\left(\int_{0}^t \mathds{1}\{\o_{\mzero}(s)=0\}ds<\epsilon t
			\right)\nn \\
			&& +\sum_{k=\left\lceil \delta t\right\rceil}^{\infty}\nu_k\,
			\P_{\k}\left(\int_{0}^t\mathds{1}\{\o_{\mzero}(s)=0\}ds<
			\epsilon t\right)\nn\\
			&\leq& \sum_{k=1}^{\left\lfloor \delta t\right\rfloor}\nu_k\,
			\P_{\k}\left(\int_{0}^t\mathds{1}\{\o_{\mzero}(s)=0\}ds<\epsilon t\right)\nn\\
			&& +\nu([\d t,\infty)). 
			\label{eq:52}
		\end{eqnarray}
		Thus, 
		\begin{equation}
			\E_{\hat{\nu}}\left(\tu\right) \leq \epsilon^{-1} + \d^{-1} \Er(\W)
			+\int_{0}^{+\infty} \sum_{k=1}^{\left\lfloor \delta t\right\rfloor}\nu_k\,\P_{\k}
			\left(\int_{0}^t\mathds{1}\{\o_{\mzero}(s)=0\}ds<\epsilon 
			t\right)dt,
			\label{3.83}
		\end{equation}
		where $\W$ is a $\nu$-distributed random variable; one readily checks that~\eqref{extracon} implies that $\W$ has a first moment.
		It remains to consider the latter summand in~\eqref{3.83}.

		For that, let us start by setting $W_0=\inf\{s>0:\o_{\mzero}(s)= 0\}$, and defining
		\begin{gather}
			Z_1=\inf\left\lbrace s>W_0:\o_{\mzero}(s)\neq 0\right\rbrace-W_0,\\
			W_1=\inf\left\lbrace s>W_0+Z_1:\o_{\mzero}(s)=0\right\rbrace-\left(W_0+Z_1\right),
			\label{eq:53-34}
		\end{gather}
		and making $Y_1=Z_1+W_1$. Note that $Z_1$ is an exponential random variable with rate $p_0$, and $W_1$ 
		is the hitting time of the origin by a $BDP(\p,\q)$ on $\N$ starting from 1; under $\P_{\mzero}$, $W_0=0$, clearly.

		For $i\geq 1$, let us suppose defined $Y_1,\ldots,Y_{i-1}$, and let us further define
		\begin{gather}
			Z_i=\inf\left\lbrace s>W_0+\sum_{j=1}^{i-1}Y_j:\o_{\mzero}(s)\neq 0\right\rbrace-\left(W_0+\sum_{j=1}^{i-1} Y_j\right),\\
			W_i=\inf\left\lbrace s>W_0+\sum_{j=1}^{i-1} Y_j+Z_i:\o_{\mzero}(s)=0\right\rbrace-\left(W_0+\sum_{j=1}^{i-1} Y_j+Z_i\right),
			\label{eq:55-56}
		\end{gather}
		and $Y_i=Z_i+W_i$. By the strong Markov property, it follows that $Z_i$ e $W_i$ are distributed as $Z_1$ e $W_1$, respectively,
		and $Z_i,W_i$, $i\geq1$ are independent, and thus $\left(Y_i\right)_{i\geq 1}$ is iid. 
		
		Now set $T_0=W_0$ and for$n\geq 1$, 
		$T_n=T_{n-1}+Y_n$. Moreover, for $t\in\R_+$, let us define $\sC_t=\sum\limits_{n=1}^{\infty}\mathds{1}\left\lbrace 
		T_n\leq t\right\rbrace$. Note that for $k\in\N$ and $a>0$, we have
		\begin{eqnarray}
			\P_{\k}\left(\int_{0}^t\mathds{1}\{\o_{\mzero}(s)=0\}ds<\epsilon t\right)&=&
			\P_{\k}\left(\int_{0}^t\mathds{1}\{\o_{\mzero}(s)=0\}ds<\epsilon t, \sC_t<\left\lfloor at \right\rfloor\right)\nn \\
			&& +~\P_{\k}\left(\int_{0}^t\mathds{1}\{\o_{\mzero}(s)=0\}ds<\epsilon t,\sC_t\geq \left\lfloor at\right\rfloor\right)\nn\\
			&\leq& \P_{\k}\left(\sC_t<\left\lfloor at \right\rfloor\right)+ \P\left(\sum_{j=1}^{\left\lfloor at\right\rfloor}Z_j<\epsilon t\right)
			\label{eq:57}
		\end{eqnarray}
		and, given $\alpha\in(0,1)$,
		\begin{align}
			\P_{\k}\left(\sC_t<\lfloor at\rfloor \right)&=\P_{\k}\left(\sC_t<\lfloor at\rfloor,T_0<\alpha t \right)+
			\P_{\k}\left(\sC_t<\lfloor at\rfloor,T_0\geq \alpha t \right)\nn\\
			&\leq \P_{\zero}\left(\sC_{(1-\alpha)t}<\lfloor at\rfloor \right)+\P_{\k}\left(T_0\geq \alpha t \right)\nn\\
			&=\P\left(\sum_{j=1}^{\left\lfloor at\right\rfloor}Y_j>(1-\alpha)t \right)+\P_{\k}\left(T_0\geq \alpha t \right).
			\label{eq:58}
		\end{align}

		By well known elementary large deviation estimates, we have that 
		\begin{equation}
			\int_0^\infty dt\,\P\!\left(\sum_{j=1}^{\left\lfloor at\right\rfloor}Z_j<\epsilon t\right)<\infty
			\label{eq:59}
		\end{equation}
		as soon as $a<p_0\epsilon$, which we assume from now on. To conclude, it then suffices to show that
		\begin{equation}
			\int_0^\infty dt\,\P\left(\sum_{j=1}^{\left\lfloor at\right\rfloor}Y_j>(1-\alpha)t \right)<\infty\,\mbox{ and }\,
			\int_0^\infty dt\,\sum_{k\geq0}\nu_k\,\P_{\k}\left(T_0\geq \alpha t \right)<\infty.
			\label{eq:61}
		\end{equation}
		The latter integral is readily seen to be bounded above by $\alpha^{-1} \Er_\nu(T_0)$, and the first condition in~\eqref{extracor} implies 
		the second assertion in~\eqref{eq:61}. The first integral in~\eqref{eq:61} can be written as
		\begin{equation}
			\int_0^\infty dt\,\P\left(\frac1{at}\sum_{j=1}^{\left\lfloor at\right\rfloor}\bar Y_j>\zeta \right),
			\label{cconv}
		\end{equation}
		where $\bar Y_j=Y_j-b$, $b=\E Y_1=\E Y_j$, $j\geq1$, $\zeta=(1-\alpha-ab)/a$. Now we have that the expression in~\eqref{cconv}  is finite
		by the Complete Convergence Theorem of Hsu and Robbins (see Theorem 1 in~\cite{HR}), as soon as $a,\alpha>0$ are close enough 
		to 0 (so that $\zeta>0$), and $W_1$ has a second moment (and thus so does $Y_1$), but this follows immediately  from the first condition in~\eqref{extracor}.
		
	\end{proof}
	
	We are now ready to state and prove the main result of this subsection.
	
	\begin{proposicao} There exists a constant $\mu\in[0,\infty)$ such that
		\label{prop:3.1}
		\begin{equation}
			\frac{\tn}{n}\to \mu \quad \P_{\zero}\text{-a.s.} \quad \textrm{as } 
			n\to\infty.
			\label{3.50}
		\end{equation}
		Furthermore,
		\begin{equation}\label{3.50a}
			\mu>0.
		\end{equation}
	\end{proposicao}
	\begin{proof}
		\noindent We divide the argument in two parts. We first construct a superadditive triangular array of random variables 
		$\lbrace \sL_{m,n}:m,n\in\N, m\leq n\rbrace$ so that  $\sL_{0,n}$ equals $\tau_n$ under $\P_{\zero}$. 
		Secondly, we verify that $\{-\sL_{m,n}:m,n\in\N, m\leq n\}$ satisfies the conditions of Liggett's version of Kingman's 
		Subadditive Ergodic Theorem, an application of which yields the result.  
		
		
		\paragraph{A triangular array of jump times} \mbox{}
		
		\smallskip
		
		Somewhat similarly as in the construction leading to Lemma~\ref{domult} (see description preceding the statement of that result),
		we construct a sequence of environments $\circo^m$, $m\geq0$, coupled to $\o$, in a {\em dominated} way (rather than {\em dominating}, 
		as in the previous case), as follows.
		
		Let $\o(0)=\zero$, and set $\circo^0=\o$. Consider now $\tau_1,\tau_2,\ldots$, the jump times of $X$, as define above.  
		For $m\geq1$, we define $(\circo^m(t))_{t\geq\tau_m}$ as a $BDP(\p,\q)$ starting from $\circo^m(\tau_m)=\zero$, coupled to $\o$ in 
		$[\tau_m,\infty)$ so that 
		\begin{equation}\label{comp4}
			\circo_\x^m(t)\leq\o_\x(t)
		\end{equation}
		for all $t\geq\tau_m$ and all $\x\in\Zd$. 
		
		Let  $\cirx^m$ be a random walk in environment $\circo^m$ starting at time $\tau_m$ from $\x_m$, with jump times determined, besides 
		$\circo^m$, the Poisson marks of $\M$ in the upper half space from $\tau_m$, in the same way as the jump times of $X$ after $\tau_m$ 
		are determined by  $(\o(t))_{t\geq\tau_m}$ and the Poisson marks of $\M$ in the upper half space from $\tau_m$, and having subsequent 
		jump destinations given by $\sx_j$, $j\geq m$. 
		Now set $\cirt^m_0=\tau_m$ and let $\cirt^m_1,\cirt^m_2,\ldots$ be the successive jump times of $\cirx^m$.
		
		Finally, for $n\geq m$, set $\sL_{m,n}=\cirt^m_{n-m}-\tau_m$. $\sL_{m,n}$ is the time $\cirx$ takes to give $n-m$ jumps. 
		Notice that $\sL_{0,n}=\tau_n$.

		\paragraph{Properties of $\{\sL_{m,n},\,0\leq m\leq n<\infty\}$}\mbox{}
		
		\smallskip
		
		We claim that the following assertions hold.
		\begin{eqnarray}
			&\sL_{0,n}\geq \sL_{0,m}+\sL_{m,n}\,\,\P_{\zero}\text{-a.s.};&\label{sad1}\\
			&\left\{\sL_{nk,(n+1)k}, n\in \N\right\} \text{ is ergodic for each }\, k\in\N;&\label{sad2}\\
			&\text{ the distribution of } \left\{\sL_{n,n+k}: k\geq 1\right\} \text{ under $\P_{\zero}$ does not depend on }\, n\in\N;&\label{sad3}\\
			& \text{there exists }\, \gamma_0<\infty \,\text{ such that }\,\E_\zero(\sL_{0,n})\leq \gamma_0 n. &\label{sad4}
		\end{eqnarray}

		\eqref{3.50} then follows from an application of  Liggett's version of Kingman's Subadditive Ergodic Theorem to $(-\sL_{m,n})_{0\leq m\leq n<\infty}$ 
		(see~\cite{Lig}, Chapter VI, Theorem $2.6$).

		\eqref{sad3}  is quite clear,  \eqref{sad2} follows immediately upon remarking that $\sL_{nk,(n+1)k}$, $n\in \N$, are, quite clearly, independent random variables,
		and~\eqref{sad4} follows readily from~\eqref{dom2} and~\eqref{fin1}. So, it remains to argue~\eqref{sad1}, which is equivalent to 
		\begin{equation}\label{sub}
			\cirt^m_{n-m}\leq\tau_n,\,0\leq m\leq n<\infty.
		\end{equation}
		
		We make this point similarly as for~\eqref{comp3}, above.
		\eqref{sub} is immediate for $m=0$. Let us fix $m\geq1$. Then~\eqref{sub} is immediate for $n=m$, and for $n=m+1$ it follows readily from the
		fact that $\circo_{\sx_m}(t)\leq\o_{\sx_m}(t)$, $t\geq\tau_m$.
		
		For the remaining cases, let $n\geq m+1$, and suppose, inductively, that $\cirt^m_{n-m}\leq\tau_n$; there are two
		possibilities for $\cirt^m_{n+1-m}$: either $\cirt^m_{n+1-m}\leq\tau_n$, in which case, clearly, $\cirt^m_{n+1-m}\leq\tau_{n+1}$, or $\cirt^m_{n+1-m}>\tau_n$; 
		in this latter case, $\tau_{n+1}$ (resp., $\cirt^m_{n+1-m}$)
		will correspond to the earliest Poisson point (of $\M$) in $\cQ'_n:=[c_{\sx_n(d)},c_{\sx_n(d)}+\vf(\o_{\sx_n}(r))_{r\geq\tau_n}$ 
		(resp., $\mathring\cQ_n:=[c_{\sx_n(d)},c_{\sx_n(d)}+\vf(\circo^m_{\sx_n}(r))_{r\geq\tau_n}$). By~\eqref{comp4} and the monotonicity of $\vf$, we have that
		$\mathring\cQ_n\supset\cQ'_n$, and it follows that  $\cirt^m_{n+1-m}\leq\tau_{n+1}$.
		
		Finally, one readily  checks from~\eqref{sad1} that $\mu\geq\Er_0(\tau_1)$; the latter expectation can be readily checked to be strictly positive, 
		and the argument is complete.
		
	\end{proof}

\end{subsection}

\begin{subsection}{Proof of the Law of Large Numbers for $X$ under $\P_{\zero}$}
	\label{sec:3.3/2}
	
	We may now prove Theorem~\ref{teo:3.1}.
	For $t\in\R_+$, let  $\mathsf{N}_t=\inf{\left\lbrace n\geq 0: \tn<t\right\rbrace}$. 
	It follows readily from Proposition~\ref{prop:3.1}  that
	\begin{equation}
		\frac{\mathsf{N}_t}{t}\to \frac{1}{\mu} ~\P_{\zero}\text{-a.s.}
		~\textrm{as $t\to \infty$}.
		\label{eq:65}
	\end{equation}
	
	It follows from~\eqref{eq:65} and the Strong Law of Large Numbers for $(\x_n)$ that
	\begin{equation}
		\frac{X(t)}{t}=\frac{\mathsf{x}_{\mathsf{N}_t}}{t}=\frac{\mathsf{x}_{\mathsf{N}_t}}{\mathsf{N}_t}
		\times\frac{\mathsf{N}_t}{t}\to \frac{\E(\xi_1)}{\mu}~\P_{\zero}\text{-a.s}
		~ \text{as $t\to\infty$}.
		\label{eq:66}
	\end{equation}

\end{subsection}

\begin{subsection}{Proof of the Central Limit Theorem for $X$ under $\P_{\zero}$}
	\label{sec:3.2}
	
	\noindent We now prove Theorem~\ref{teo:3.2}. Let $\ga=1/\mu$, and write
	\begin{equation}\label{subs}
		\frac{X(t)}{\sqrt{\ga t}}=
		\frac{\sx_{\sN_t}-\sx_{\lfloor \ga t\rfloor}}{\sqrt{\ga t}}+ \frac{\sx_{\lfloor \ga t\rfloor}}{\sqrt{\ga t}}.
	\end{equation}

	By the Central Limit Theorem obeyed by $(\sx_n)$, we have that, under $\P$, as $t\to\infty$,
	\begin{equation}\label{clt}
		\frac{\sx_{\lfloor \ga t\rfloor}}{\sqrt{\lfloor \ga t\rfloor}}\Rightarrow N_d(\zero,\Sigma).
	\end{equation}

	We now claim that the first term on the right hand side of~\eqref{subs} (after multiplication by $\ga$) vanishes in probability as $t\to\infty$ under $\Pz$.
	Indeed, let us write $\xi_k=\left(\xi_{k,1},\ldots,\xi_{k,d}\right)$, $k\in\N$. Given $\epsilon>0$, let us set $\d=\eps^3$; we have that

	\begin{equation}\label{decomp}
		\P_{\zero}\left(\left\lVert \frac{\sx_{\sN_t}-\sx_{\lfloor \ga t\rfloor}}{\sqrt{t}}\right\rVert>\epsilon\right)\leq
		\P_{\zero}\left(\left\lVert \sx_{\sN_t}-\sx_{\lfloor \ga t\rfloor}\right\rVert>\eps\sqrt{t},\,\left|{\sN_t}-\gamma t\right|<{\d t}\right)+
		\P_{\zero}\left(\left|{\sN_t}-\gamma t\right|\geq{\d t}\right).
	\end{equation}
	By~\eqref{eq:65}, it then suffices to consider the first term on the right hand side of~\eqref{decomp}, which may be readily seen to be bounded above by
	\begin{equation}
		\sum_{i=1}^d\left\{\P_{\zero}\left(\max_{0\leq\ell\leq\d t}\left| \sum_{k=\ga t-\ell}^{\ga t}\xi_{k,i}\right|>\epsilon\sqrt{t}\right)
		+ 
		\P_{\zero}\left(\max_{0\leq\ell\leq\d t}\left| \sum_{k=\ga t}^{\ga t+\ell}\xi_{k,i}\right|>\epsilon\sqrt{t}\right)\right\}
		\leq 3\,\mathrm{Tr}\!\left(\Sigma\right)\eps,
		\label{eq:3.97}
	\end{equation}
	\noindent where we have used Kolmogorov's Maximal Inequality in the latter passage; the claim follows since $\eps$ is arbitrary. 
	And the CLT follows readily from the claim and~\eqref{clt}.
	
	\medskip
	
	\begin{observacao}\label{ext_clt}
		A meaningful extension of our arguments for the above CLT to the non mean zero case would require understanding the fluctutations of $(\sN_t)$, and their dependence to those of a centralized $X$, issues that we did not pursue for the present article, even though they are most probably treatable by a regeneration argument (possibly dispensing with the domination requirements of our argument for the mean zero case, in particular that $\vf$ be decreasing).
		
		Another extension is to prove a functional CLT; for the mean zero case treated above, that, we believe, requires no new ideas, and thus we refrained to present a standard argument to that effect (having already gone through standard steps in our justifications for the LLN and CLT for $X$).
	\end{observacao}

	\begin{observacao}\label{relax}
		It is quite clear from our arguments that all that we needed to have from our conditions on $\p,\q$ is the validity of both conditions in~\eqref{extracor}, and 
		thus we may  possibly relax to some extent~\eqref{extracon}, and certainly other conditions imposed on $\p,\q$ (in the paragraph of~\eqref{erg}), with the same approach, but we have opted for simplicity and cleanness, within a measure of generality.
	\end{observacao}

	\begin{observacao}\label{rever}
		For the proof of Lemma~\ref{3.17}, a mainstay of our approach, we relied on 
		the reversibility of the birth-and-death process,  the positivity of $d_n$, and the increasing monotonicity of $\psi$; see the upshot of the paragraph of~\eqref{3.27a}. It is natural to think of extending the argument for other reversible ergodic Markov processes on $\N$; one issue for longer range cases is the positivity of $d_n$; there should be examples of long range reversible ergodic Markov processes on $\N$ where positivity of $d_n$ may be ascertained by a coupling argument, and we believe we have worked out such an example, but it looked too specific to warrant a more general formulation of our results (and the extra work involved in such an attempt), so again we felt content in presenting our approach in the present setting.
	\end{observacao}
	
	\begin{observacao}\label{alt-lln}
		Going back to the construction leading to Lemma~\ref{domult}, for $0\leq m\leq n$, let $\brel_{m,n}$ denote the time $\o^m$ takes to give $n-m$ jumps.
		Then it follows from the properties of $\o,\o^m$, $m\geq0$, as discussed in the paragraphs preceding the statement of Lemma~\ref{domult}, that
		$\{\brel_{m,n},\,0\leq m\leq n<\infty\}$ is a {\em subadditive} triangular array, and a Law of Large Numbers for $\tau_n$ under $\P_{\hat\nu}$ would follow, once we establish ergodicity of  $\{\brel_{nk,(n+1)k}, n\in \N\}$, other conditions for the application of the Subadditive Ergodic Theorem being readily seen to hold. 
		This would require a more susbstantial argument than for the corresponding result for $\{\sL_{nk,(n+1)k}, n\in \N\}$, made briefly above (in the second paragraph
		below~\eqref{sad4}), since independence is lost. Perhaps a promising strategy would be one similar to that which we undertake in next section, to the same effect; see Remark~\ref{rem:erg}. For this, if for nothing else, we refrained from pursuing  this specific point in this paper.
	\end{observacao}
	
	\begin{observacao}\label{vfzero}
		The restriction of positivity of $\vf$, made at the beginning, is not really crucial in our approach. It perhaps makes parts of the arguments clearer, but our approach works if we allow for $\vf(n)=0$ for $n\geq n_0$ for any given $n_0\geq1$ --- in this case, we note, the auxiliary process $\Y$ introduced in the proof of Lemma~\ref{lema:3.2} is a birth-and-death process on $\{0,\ldots,n_0-1\}$.
	\end{observacao}

\end{subsection}
%



\section{Other initial conditions}
\label{ext}

\setcounter{equation}{0}

In this section we extend Theorem~\ref{teo:3.2} to other (product) initial conditions. In this and in the next section, we will assume for simplicity that  the BD process environments are homogeneous, i.e., $p_n\equiv p$, with $p\in(0,1/2)$. In this context, we use the notation $BDP(p,q)$ for the process, where $q=1-p$.  We hope that the arguments developed for the inhomogeneous case, as well as subsequent ones, are sufficiently convincing that this may be relaxed --- although we do not pretend to be able to propose optimal or near optimal conditions for the validity of any of the subsequent results.

As we will see below, our argument for this extension does not go through a LLN for the position of the particle, as it did in the previous section, 
we do not discuss an extension for the LLN, rather focusing on the CLT.\footnote{But the same line of argumentation below may be readily 
	seen to yield a LLN, under the same conditions.}

We will as before assume that the initial condition for the environment is product, given by $\hat\mu_0=\bigotimes\limits_{\x\in\Zd}\mu_{\x,0}$, 
and we will further assume that $\mu_{\x,0}\preceq\bar\mu$, with $\mu$ a probability measure on $\N$ with an exponentially decaying tail, i.e.,
there exists a constant $\beta>0$ such that
\begin{equation}\label{expdec}
	\bar\mu([n,\infty))\leq\text{const } e^{-\beta n}
\end{equation}
for all $n\geq0$. Notice that this includes $\hat\nu$, in the present homogeneous BDP case. Again, it should hopefully be quite clear from our arguments that these conditions can be relaxed both in terms of the homogeneity of $\bar\mu$, as the decay of its tail, but we do not seek to do that presently, or to suggest optimal or near optimal conditions.

Our strategy is to first couple the environment starting from $\hat\mu_0$ to the one starting from $\zero$, so that for each $\x\in\Zd$, each respective BD process evolves independently one from the other until they first meet, after which time they coalesce forever.

One natural second step is to couple two versions of the random walks, one starting from each of the two coupled environments in question, so that they jump together when they are at the same point at the same time, and see the same environment. One quite natural way to try and implement such a strategy is to have both walks have the same embedded chains, and show that they will (with high probability) eventually meet at a time at and after which they only see the same environments. Even though this looks like it should be true, we did not find a way to control the distribution of the environments seen by both walks in their evolution (in what might be seen as a {\em game of pursuit}) in an effective way. 

So we turned to our actual subsequent strategy, which depends on the dimension (and requires different further conditions on $\pi$, the distribution of $\xi$, in $d\geq2$). In $d\leq2$, we modify the strategy proposed in the previous paragraph by letting the two walks evolve independently when separated, and relying on recurrence to ensure that they will 
meet in the afore mentioned conditions; there is a technical issue arising in the latter point for general $\pi$ (within the conditions of 
Theorem~\ref{teo:3.2}), which we resolve by invoking a result in the literature, which is stated for $d=1$ only, so for $d=2$ we need to restrict $\pi$ to be symmetric. See Remark~\ref{symm} below.

In $d\geq3$, we of course do not have recurrence, but, rather, transience, and so we rely on this, instead, to show that our random walk  will eventually find itself  in a cut point of its trajectory such that the environment along its subsequent trajectory  is coalesced with a suitably coupled environment starting from $\zero$; this allows for a comparison to the situation of Theorem~\ref{teo:3.2}.
The argument requires the a.s.~existence of infinitely many cut points of $(\xn)$, and, to ascertain that, we rely on the literature, which states boundedness of the support of $\pi$ as a sufficient condition (but no symmetry).

\begin{teorema} [Central Limit Theorem for $X$] \label{gclt}\mbox{}
	
	Under the same conditions of Theorem~\ref{teo:3.2}, and assuming the conditions on $\hat\mu_0$ stipulated in the paragraph of~\eqref{expdec} above hold, then we have that for $\Pm_{\hat{\mu}_{{0}}}$-a.e. $\o$
	\begin{equation}
		\frac{X(t)}{\sqrt{ t/\mu}}\Rightarrow N_d(\zero,\Sigma) \,\mbox{ under }\, P^{^{\o}}, 
		\label{eq:75ext}
	\end{equation}
	\label{teo:3.2ext}
	%
	provided the following extra conditions on $\pi$ hold, depending on $d$: in $d=1$, no extra condition; in $d=2$, $\pi$ is symmetric; 
	in $d\geq3$, $\pi$ has bounded support.
\end{teorema}

We present the proof of Theorem~\ref{teo:3.2ext} in two arguments, spelling out the above broad descriptions,  in two subsequent subsections, one for $d\leq2$, and another one for $d\geq3$. We first state and prove a lemma which enters 
both arguments, concerning successive coalescence of coupled versions of the environments, one started from $\zero$, and the other from $\hat\mu_0$, over certain times related to displacements of $(\xn)$.

Consider two coalescing versions of the environment, $\circo$ and $\o$, the former one starting from $\zero$, and the latter starting from $\hat\mu_0$ as above, such that $\circo_\x(t)\leq\o_\x(t)$ for all $\x$ and $t$, and for $\x\in\Zd$, let $\sT_{\x}$ denote the coalescence time of $\circo_\x$ and $\o_\x$, i.e.,
\begin{equation}
	\sT_{\x}=\inf\left\lbrace s>0:\circo_{\x}(s)=\o_{\x}(s)\right\rbrace.
\end{equation}

Now let $\cirx$ and $X$ be versions of the random walks on $\Zd$ in the respective environments, both starting from $\zero(\in\Zd)$.
Let us suppose, for simplicity, that they have the same embedded chain $(\xn)$.
For $n\in\N$, let $\B_n$ denote $\{-2^n,-2^n+1,\ldots,2^n-1,2^n\}^d$, let $\circh_n$ (resp., $\cH_n$) denote the hitting time of $\Zd\setminus\B_n$ by $\brex$ (resp., $X$), and consider the event $\circa_n$ (resp., $\mA_n$) that $\sT_{\x}\leq\circh_n$ (resp., $\sT_{\x}\leq\cH_n$) for all $\x\in\B_{n+1}$.
Let also $\h_n$ denote the hitting time of $\Zd\setminus\B_n$ by $(\xn)$.

\begin{lema} \label{coupenv}

	\begin{equation}\label{coupenv1}
		\P_{\zero}(\circa_n^c\text{ infinitely often})=\P_{\hat{\mu}_{{0}}}(\mA_n^c\text{ infinitely often})=0
	\end{equation}
	
\end{lema}

\begin{proof}
	Under our conditions, the argument is quite elementary, and for this reason we will be rather concise.  
	Let us first point out that both $\circh_n$ and $\cH_n$ are readily seen to be bounded from below stochastically by 
	$\bh_n:=\sum_{i=1}^{\h_n}\cE_i$, where $\cE_1,\cE_2,\ldots$ are iid standard exponential random variables, 
	which independent of $\h_n$ and of $\circo$ and $\o$.
	
	It follows readily from Kolmogorov's Maximal Inequality that for all $n\in\N$
	\begin{equation}\label{kmi}
		\P(\h_n\leq 2^n)=\P\Big(\max_{1\leq i\leq 2^n}\|\sx_i\|> 2^n\Big)\leq\text{const } 2^{-n},
	\end{equation}
	and by the above mentiond domination and elementary well known large deviation estimates, we find that
	\begin{equation}\label{coupenv2}
		\Pz(\circh_n\leq 2^{n-1})\vee\P_{\hat{\mu}_{{0}}}(\cH_n\leq 2^{n-1})\leq\P(\bh_n\leq 2^{n-1})\leq\text{const } 2^{-n}.
	\end{equation}
	We henceforth treat only the first probability in~\eqref{coupenv1}; the argument for the second one is identical.
	
	The probability of the event that $\circh_n\leq 2^{n-1}$ and $\sT_{\x}>\cH_n$ for some $x\in\B_{n+1}$ is bounded above by
	\begin{equation}\label{coupenv3}
		\text{const }2^{dn}\,\Pr(\sT_{\zero}> 2^{n-1}).
	\end{equation}
	
	It may now be readily checked that $\sT_{\zero}$ is stochastically dominated by the hitting time of the origin by a simple symmetric random walk on $\Z$  
	in continuous time with homogeneous jump rates equal to 1, with probability $p$ to jump to the left, initially distributed as $\bar\mu$. Thus, given $\d>0$
	\begin{equation}\label{coupenv4}
		\Pr(\sT_{\zero}> 2^{n-1})\leq\bar\mu([\d2^n,\infty))+\Pr\Big(\sum_{i=1}^{\d2^n}H_i>2^{n-1}\Big),
	\end{equation}
	where $H_1,H_2$ are iid random variables distributed as the hitting time of the origin by a simple symmetric random walk on $\Z$  
	in continuous time with homogeneous jump rates equal to 1, with probability $p$ to jump to the left, starting from 1. 
	$H_1$ is well known to have a positive exponential moment; it follows from elementary large deviation estimates that we may choose
	$\d>0$ such that the latter term on the right hand side of~\eqref{coupenv4} is bounded above by const $e^{-b2^n}$ for some constant 
	$b>0$ and all $n$. Using this bound, and substituting~\eqref{expdec}  in~\eqref{coupenv4}, we find that
	\begin{equation}\label{coupenv5}
		\Pr(\sT_{\zero}> 2^{n-1})\leq\text{const }e^{-b'2^n}
	\end{equation}
	for some $b'>0$ and all $n$, and~\eqref{coupenv1} upon a suitable use of the Borel-Cantelli Lemma.
	
\end{proof}

\begin{observacao}\label{rem:erg}
	As vaguely mentioned in Remark~\ref{alt-lln} at the end of the previous section,
	a seemingly promising strategy for establishing the ergodicity of $\{\brel_{nk,(n+1)k}, n\in \N\}$
	would be to approximate an event of $\F^+_{m'}$, the $\s$-field generated by $\{\brel_{nk,(n+1)k}, n\geq m'\}$, by one generated by a version of
	an environment starting from $\zero$ at time $\brel_{0,mk}$, coupled to the original environment in a coalescing way as above, with suitable
	couplings of the jump times and destinations, with fixed $m\in\N_*$ and $m'\gg m$. Ergodicity would follow by the independence of the latter
	$\s$-field and $\F^-_{m}$, the $\s$-field generated by $\{\brel_{(n-1)k,nk}, 1\leq n\leq m\}$. 
	We have not attempeted to work this idea out in detail; if we did, it looks as though we might face the same issues arising in the extension of the CLT,
	as treated in the present section, thus possibly not yielding a better result than Theorem~\ref{teo:3.2ext}.
\end{observacao}

\begin{subsection}{Proof of Theorem~\ref{teo:3.2ext} for $d\leq2$}
	\label{12d}
	
	We start by fixing the coalescing environments $\circo$ and $\o$,  as above, and considering two independent random walks, denoted 
	$\cirx$ and $X'$ in the respective environments $\circo$ and $\o$. The jump times of $\cirx$ and $X'$ are obtained from $\cirm$ and 
	$\M'$, as in the original construction of our model, where $\cirm$ and $\M'$ are independent versions of $\M$. 
	
	For the jump destinations of $\cirx$ and $X'$,  we will change things a little, and consider independent families $\cirxi=\{\cirxi_\sz,\,\sz\in\cirm\}$ and $\xi'=\{\xi'_\sz,\,\sz\in\M'\}$
	of independent versions of $\xi_1$. The jump destination of $\cirx$ at the time corresponding to an a.s.~unique point $\sz$ of $\cirm$ is then 
	given by $\cirxi_\sz$, and correspondingly for $X'$.

	Let $\sfD=\big(\sfD(s):=\cirx(s)-X'(s), s\geq0\big)$, which is clearly a continuous time jump process, and consider the embedded chain of $\sfD$, denoted
	$\sd=\left(\sd_n\right)_{n\in \N}$. We claim that under the conditions of Theorem~\ref{teo:3.2ext} for $d\geq2$, $\sd$ is recurrent, that is, 
	it a.s.~returns to the origin infinitely often. 
	
	Before justifying the claim, let us indicate how to reach the conclusion of the proof of Theorem~\ref{teo:3.2ext} for $d\leq2$ from this. We consider the 
	sequence of  return  times of $\sfD$ to the origin, i.e., $\tsi_0=0$, and for $n\geq1$, 
	\begin{equation}\label{meet}
		\tsi_n=\inf\big\{s>\tsi_{n-1}:\, \sfD(s)=0\text{ and } \sfD(s-)\ne0\big\}.
	\end{equation}
	It may be readily checked, in particular using the recurrence claim, that this is an infinite  sequence of a.s.~finite stopping times given $\circo,\o$, such that
	$\tsi_n\to\infty$ as $n\to\infty$.
	
	Then, for each $n\in\N$, we define a version of $X'$, denoted $X_n$, coupled to $\cirx$ and $X'$ as follows: $X_n(s)=X'(s)$ for $s\leq\tsi_n$, and for 
	for $s>\tsi_n$, the jump times and destinations of $X_n$ are defined from $\o$ as before, except that we replace the Poisson marks of $\M'$ in the half space above $\tsi_n$ by the corresponding marks of $\cirm$, and we use the corresponding jump destinations of $\cirxi$. It may be readily checked that $X_n$ is a version of $X'$, and that starting at $\tsi_n$, and as long as $X_n$ and $\cirx$ see the same respective environments, they remain together.

	It then follows from Lemma~\ref{coupenv} that there exists a finite random time $N$ such that $\cirx(t)$ and $X'(t)$ each sees only coupled environments
	for $t>N$, and thus so do $\cirx(t)$ and $X_n(t)$ for $t>\tsi_n>N$. It then follows from the considerations above that given $\circo,\o$, $n\in\N$ and  $x\in\R$
	\begin{eqnarray}\nn
		&\Big|P\Big(\frac{X'(t)}{\sqrt t}<x\Big)-P\Big(\frac{\cirx(t)}{\sqrt t}<x\Big)\Big|
		=\Big|P\Big(\frac{X_n(t)}{\sqrt t}<x\Big)-P\Big(\frac{\cirx(t)}{\sqrt t}<x\Big)\Big|&\\
		&\leq P\big(\{t>\tsi_n>N\}^c\big)\leq P(\tsi_n\geq t)+P(N\geq \tsi_n),&
		\label{clt1}
	\end{eqnarray}
	and it follows that the limsup as $t\to\infty$ of the left hand side of~\eqref{clt1} is bounded above by the latter probability in the same expression. 
	The result (for $X'$) follows since (it does for $\cirx$, by Theorem~\ref{teo:3.2}, and) $n$ is aribitrary.
	
	In order to check the recurrence claim, notice that if $\pi$, the distribution of $\xi_1$, is symmetric, then $\sd$ is readily seen to be a discrete time random walk on $\Zd$ with jump distribution given by $\pi$, and the claim folllows from well known facts about mean zero random walks with finite second moments for $d\leq2$. 
	This completes the argument for Theorem~\ref{teo:3.2ext} for $d=2$.
	
	For $d=1$ and asymmetric $\pi$, $\sd$ is no longer Markovian, but we may resort to Theorem 1 of~\cite{RHG91} to justify the claim as follows.
	Let us fix a realization of $\circo$, $\o$, $\cirm$ and $\M'$ (such that no two marks in $\cirm\cup\M'$ have the same time coordinate, which is of course an event of full probability). 
	Let us now dress $\sd$ up as a {\em controlled random walk (crw)}  (conditioned on $\circo$, $\o$, $\cirm$ and $\M'$), in the language of~\cite{RHG91}; see paragraph before the statement of Theorem 1 therein. 
	
	There are two kinds of jump distributions for $\sd$ ($p=2$, in the notation of~\cite{RHG91}):
	$F_1$ denotes the distribution of $\xi_1$, and $F_2$ denotes the distribution of $-\xi_1$. 
	In order to conform to the set up of~\cite{RHG91}, we will also introduce two independent families of (jump) iid random variables 
	(which will in the end not be used), namely, $\brexi=\{\brexi_\sz,\,\sz\in\cirm\}$ and $\xi''=\{\xi''_\sz,\,\sz\in\M'\}$, independent of, but having the same marginal distributions as, $\cirxi$ (and $\xi'$).
	
	Let us see how the choice between each of the two distributions is made at each step of $\sd$. This is done using the indicator functions 
	$\psi_n$, introduced and termed in~\cite{RHG91} the {\em choice of game at time} $n\geq1$, inductively, as follows. 
	
	Given $\circo$, $\o$, $\cirm$ and $\M'$, let $\zeta_1$ denote the earliest point of $\cirn_\zero \cup\cN'_\zero$, where $\cirn_\x$, $\cN'_\x$, $\x\in\Zd$, are defined from $(\circo,\cirm)$ and $(\o,\M')$, respectively, as $\Nx$ was defined from $(\o,\M)$ at the beginning of Section~\ref{mod},
	and let $\eta_1$ denote the time coordinate of $\zeta_1$, and set $\psi_1= 1+\oone\{\zeta_1\in\cN'_\zero\}$, and 
	\begin{equation}\label{psi1}
		X_1^i:=
		\begin{cases}
			\,\,\,\,\cirxi_{\zeta_1},&\text{ if } \psi_1=1\text{ and } i=1,\\
			-\brexi_{\zeta_1},&\text{ if } \psi_1=1\text{ and } i=2,\\
			\,\,\,\,\xi''_{\zeta_1},&\text{ if } \psi_1=2\text{ and } i=1,\\
			-\xi'_{\zeta_1},&\text{ if } \psi_1=2\text{ and } i=2.
		\end{cases}	
	\end{equation}
	Notice that $X_1^1$ and $X_1^2$ are independent and distributed as $F_1$ and $F_2$, respectively, and that a.s.
	\begin{eqnarray}\label{x121}
		\cirx(\eta_1)&=&X_1^{\psi_1}\oone\{\psi_1=1\}+\cirx(0)\oone\{\psi_1=2\},\\
		X'(\eta_1)&=&X_1^{\psi_1}\oone\{\psi_1=2\}+X'(0)\oone\{\psi_1=1\}.
	\end{eqnarray} 
	
	For $n\geq2$, having defined $\zeta_j$, $\eta_j$, $\psi_j$, $X_j^i$, $j<n$, $i=1,2$, let
	$\zeta_n$ denote the earliest point of $\cirn_{\cirx(\eta_{n-1})}(\eta_{n-1})\cup \cN'_{X'(\eta_{n-1})}(\eta_{n-1})$,
	where for $\x\in\Zd$ and $t\geq0$, $\cirn_\x(t)$, $\cN'_\x(t)$ denote the points of $\cirn_\x$, $\cN'_\x$ with time coordinates above $t$,
	respectively. 
	
	Let now $\eta_n$ denote the time coordinate of $\zeta_n$, and set 
	$\psi_n= 1+\oone\{\zeta_n\in\cN'_{X'(\eta_{n-1})}(\eta_{n-1})\}$, and
	\begin{equation}\label{psin}
		X_n^i:=
		\begin{cases}
			\,\,\,\,\cirxi_{\zeta_n},&\text{ if } \psi_1=1\text{ and } i=1,\\
			-\brexi_{\zeta_n},&\text{ if } \psi_1=1\text{ and } i=2,\\
			\,\,\,\,\xi''_{\zeta_n},&\text{ if } \psi_1=2\text{ and } i=1,\\
			-\xi'_{\zeta_n},&\text{ if } \psi_1=2\text{ and } i=2.
		\end{cases}	
	\end{equation}
	Notice that $\{X_j^i;\,1\leq j\leq n,\,i=1,2\}$ are independent and $X_j^1$ and $X_j^2$ are distributed as $F_1$ and $F_2$, respectively,  for all $j$. Morever,  a.s.  
	\begin{eqnarray}\label{x12n}
		\cirx(\eta_n)&=&X_n^{\psi_n}\oone\{\psi_n=1\}+\cirx(\eta_{n-1})\oone\{\psi_n=2\},\\
		X'(\eta_n)&=&X_n^{\psi_n}\oone\{\psi_n=2\}+X'(\eta_{n-1})\oone\{\psi_n=1\}.
	\end{eqnarray}
	
	We then have that for $n\geq1$, $\sd_n=\sum_{j=1}^nX^{\psi_j}_j$. 
	One may readily check that (given $\circo$, $\o$, $\cirm$ and $\M'$) $\sd$ is a crw in the set up of Theorem 1 of~\cite{RHG91}, 
	an application of which readily yields the claim, and the proof of Theorem~\ref{teo:3.2ext} for $d\leq2$ is complete.

\end{subsection}

\begin{observacao}\label{symm}
	We did not find an extension of the above mentioned theorem of~\cite{RHG91} to $d=2$, 
	or any other way to show recurrence of $\left(\sd_n\right)$ for general asymmetric $\pi$
	within the conditions of Theorem~\ref{teo:3.2ext}.
\end{observacao}

\begin{subsection}{Proof of Theorem~\ref{teo:3.2ext} for $d\geq3$}
	\label{3+d}
	
	We now cannot expect to have recurrence of $\sd$, quite on the contrary, but transience suggests that we may have enough of a regeneration scheme, and we pursue precisely this idea, in order to implement which,  we resort to {\em cut times} of the trajectory of $(\xn)$, to ensure the existence of infinitely many of which, we need to restrict to boundedly supported $\pi$'s.
	
	We will be rather sketchy in this subsection, since the ideas are all quite simple and/or have appeared before in a similar guise.
	
	We now discuss a key concept and ingredient of our argument: cut times for $\sx=(\xn)$. First some notation: 
	for $i,j\in\N$, $i\leq j$, let $\sx[i,j]:=\bigcup\limits_{k=i}^j\{\sx_k\}$,   and $\sx[i,\infty):=\bigcup\limits_{l=1}^\infty\sx[i,l]$, and set	
	\begin{equation}
		\sK_1=\inf{\lbrace n\in\N:\sx[0,n]\cap \sx[n+1,\infty)=\emptyset\rbrace},
	\end{equation}	
	\noindent and, recursively, for $\ell\geq 2$,	
	\begin{equation}
		\sK_\ell:=\inf{\lbrace n>\sK_{\ell-1}:\sx[0,n]\cap \sx[n+1,\infty)=\emptyset\rbrace}.
	\end{equation}	
	\noindent $(\sK_\ell)_{\ell\in\N_*}$ is a sequence of {\em cut times} for $(\xn)$; under our conditions, it is ensured to be an a.s.~well defined infinite sequence of finite entries, according to Theorem 1.2 of~\cite{JP97}.
	
	We will have three versions of the environment coupled in a coalescent way, as above, with different initial conditions: $\circo$, starting from $\zero$;
	$\o$, starting from $\hat\mu_0$; and $\tio$, starting from $\hat\nu$; in particular, we have that $\circo_\x(t)\leq\o_\x(t),\tio_\x(t)$ for all $\x\in\Zd$ and $t\geq0$. We may suppose that the initial conditions of $\o$ and $\tio$ are independent.
	
	We now consider several coupled versions of versions of our random walk, starting with two: $X$, in the environment $\o$, as in the statement of Theorem~\ref{teo:3.2ext}; and $\cirx$, in the environment $\circo$. $X$ and $\cirx$ are constructed from
	the same $\sx$ and $\sV$, following the alternative construction of  Subsection~\ref{sec:2.2}. Let $\vs_\ell$ and $\mvs_\ell$ be the time $X$ and $\cirx$ take to give $\sK_\ell$ jumps, respectively. It may be readily checked, similarly as in Section~\ref{conv} --- see~(\ref{comp3}, \ref{sub}) ---, from the environmental monotonicity pointed to in the above paragraph and the present construction of $X$ and $\cirx$, that $\mvs_\ell\leq\vs_\ell$ for all $\ell\in\N_*$.
	
	Finally, for each $\ell\in\N_*$, we consider three modifications of $\cirx$ and $X$, namely, $\cirx_\ell$, $X_\ell$ and $X'_\ell$, defined as follows:
	\begin{equation}\label{chex}
		\cirx_\ell(t)=
		\begin{cases}
			\hspace{2cm}\cirx(t),&\text{ for } t\leq\mvs_\ell,\\
			\text{evolves in the environment } \tio,&\text{ for } t>\mvs_\ell;
		\end{cases}	
	\end{equation}
	\begin{equation}\label{xp}
		X_\ell(t)=
		\begin{cases}
			\hspace{2cm}X(t),&\text{ for } t\leq\vs_\ell,\\
			\text{evolves in the environment } \tio,&\text{ for } t>\vs_\ell;
		\end{cases}	
	\end{equation}
	\begin{equation}\label{xpp}
		X'_\ell(t)=
		\begin{cases}
			\hspace{2cm}X(t),&\text{ for } t\leq\vs_\ell,\\
			\text{evolves in the environment } \tio(\cdot-\vs_\ell+\mvs_\ell),&\text{ for } t>\vs_\ell.
		\end{cases}	
	\end{equation}
	
	Let $U$ denote the first time after which $\cirx$ and $X$ see the same environments $\circo,\o,\tio$ (from where they stand at each subsequent time).
	Lemma~\ref{coupenv} ensures that $U$ is a.s.~finite. Let us consider the event $\alt:=\{t>\vs_\ell>U\}$. 
	It readily follows that in $\alt$
	\begin{equation}\label{eqs1}
		\cirx(t)=\cirx_\ell(t)=X'_\ell(t+\vs_\ell-\mvs_\ell)\text{ and }X(t)=X_\ell(t).
	\end{equation}

	Given $\circo,\o,\tio$, let $P^{\,\circo,\o,\tio}$ denote the probability measure underlying our coupled random walks.  Since $\hat\nu$ is invariant for the environmental BD processes, it follows readily from our construction that $P^{\,\circo,\o,\tio}(X_\ell\in\cdot)$ and $P^{\,\circo,\o,\tio}(X'_\ell\in\cdot)$
	have the same distribution (as random probability measures). 
	
	For $R=(-\infty,r_1)\times\cdots\times(-\infty,r_d)$ a semi-infinite open hyperrectangle of $\R^d$, we have that
	\begin{eqnarray}\nn
		&\big|P^{\,\circo,\o,\tio}\big(X(t)\in R\sqrt{\ga t}\,\big)-P^{\,\circo,\o,\tio}\big(X_\ell(t)\in R\sqrt{\ga t}\,\big)\big|&\\\label{eqs2}
		&\leq P^{\,\circo,\o,\tio}(\alt^c)
		\leq P^{\,\circo,\o,\tio}(\vs_\ell\geq t)+P^{\,\circo,\o,\tio}(U\geq\vs_\ell)&
	\end{eqnarray}
	---  as before, $\ga=1/\mu$; see statement of  Theorem~\ref{teo:3.2ext} ---, and it follows that 
	\begin{equation}\label{eqs2a}
		\limsup_{\ell\to\infty}\limsup_{t\to\infty}
		\big|P^{\,\circo,\o,\tio}\big(X(t)\in R\sqrt{\ga t}\,\big)-P^{\,\circo,\o,\tio}\big(X_\ell(t)\in R\sqrt{\ga t}\,\big)\big|=0
	\end{equation}
	for a.e.~$\circo,\o,\tio$.
	
	Similarly, we find that for a.e.~$\circo,\o,\tio$,
	\begin{equation}\label{eqs3}
		\limsup_{\ell\to\infty}\limsup_{t\to\infty}
		\big|P^{\,\circo,\o,\tio}\big(X'_\ell(t)\in R\sqrt{\ga t}\,\big)-P^{\,\circo,\o,\tio}\big(\cirx((t-\d_\ell)^+)\in R\sqrt{\ga t}\,\big)\big|=0,
	\end{equation}
	where $\d_\ell=\vs_\ell-\mvs_\ell$.
	
	Now letting $\blt$ denote the event $\big\{\|\cirx((t-\d_\ell)^+)-\cirx(t)\|\leq\eps\sqrt{\ga t}\big\}$, where $\eps>0$, we have that
	\begin{eqnarray}\nn
		&\big|P^{\,\circo,\o,\tio}\big(\cirx((t-\d_\ell)^+)\in R\sqrt{\ga t}\,\big)-P^{\,\circo,\o,\tio}\big(\cirx(t)\in R\sqrt{\ga t}\,\big)\big|&\\\label{eqs4}
		&\leq P^{\,\circo,\o,\tio}\big(\cirx(t)\in(R^+_\eps\setminus R^-_\eps)\sqrt{\ga t}\,\big)+ P^{\,\circo,\o,\tio}\big(\blt^c\big),&
	\end{eqnarray}
	where $R^\pm_\eps=(-\infty,r_1\pm\eps)\times\cdots\times(-\infty,r_d\pm\eps)$.
	
	We now claim that for all $\ell\in\N_*$ and $\eps>0$
	\begin{equation}\label{eqs5}
		\limsup_{t\to\infty}P^{\,\circo,\o,\tio}\big(\blt^c\big)=0
	\end{equation}
	for a.e.~$\circo,\o,\tio$.
	
	It then follows from~(\ref{eqs3}, \ref{eqs4}, \ref{eqs5}) and Theorem~\ref{teo:3.2} that for $\eps>0$
	\begin{equation}\label{eqs6}
		\limsup_{\ell\to\infty}\limsup_{t\to\infty}
		\big|P^{\,\circo,\o,\tio}\big(X'_\ell(t)\in R\sqrt{\ga t}\,\big)-\Phi(R)\big|\leq\Phi(R^+_\eps\setminus R^-_\eps)
	\end{equation}
	for a.e.~$\circo,\o,\tio$, where $\Phi$ is the $d$-dimensional centered Gaussian probability measure with covariance matrix $\Sigma$. 
	Since $\eps$ is arbitrary, and the left hand side of~\eqref{eqs6} does not depend on $\eps$,  we find that it vanishes for a.e.~$\circo,\o,\tio$.
	
	
	From the remark in the paragraph right below~\eqref{eqs1}, we have that $P^{\,\circo,\o,\tio}(X_\ell(t)\in R\sqrt{\ga t})$ is distributed as 
	$P^{\,\circo,\o,\tio}\big(X'_\ell(t)\in R\sqrt{\ga t}\,\big)$; it follows that
	\begin{equation}\label{eqs7}
		\limsup_{\ell\to\infty}\limsup_{t\to\infty}\big|P^{\,\circo,\o,\tio}\big(X_\ell(t)\in R\sqrt{\ga t}\,\big)-\Phi(R)\big|=0
	\end{equation}
	for a.e.~$\circo,\o,\tio$, and it follows from~\eqref{eqs2a} that
	\begin{equation}\label{eqs8}
		\limsup_{t\to\infty}\big|P^{\,\circo,\o,\tio}\big(X(t)\in R\sqrt{\ga t}\,\big)-\Phi(R)\big|=0
	\end{equation}
	for a.e.~$\circo,\o,\tio$, which is the claim of Theorem~\ref{teo:3.2ext}.
	
	In order to complete the proof, it remains to establish~\eqref{eqs5}. For that, we first note that 
	\begin{equation}\label{bd1}
		\|\cirx((t-\d_\ell)^+)-\cirx(t)\|=\Big\|\sum_{i=\sN_{(t-\d_\ell)^+}}^{\sN_t}\xi_i\Big\|\leq K \big(\sN_t-\sN_{(t-\d_\ell)^+}\big),
	\end{equation}
	where $K$ is the radius of the support of $\pi$, and $\sN_t$, we recall from Subsection~\ref{sec:3.3/2}, 
	counts the jumps of $\cirx$ up to time $t$. Thus, the probability on the left hand side of~\eqref{eqs5} is bounded above by
	\begin{equation}\label{bd2}
		P^{\,\circo,\o,\tio}\big(\sN_t-\sN_{(t-u)^+}>\eps K^{-1}t\big)+P^{\,\circo,\o,\tio}\big(\d_\ell>u\big),
	\end{equation}
	where $u>0$ is arbitrary. 
	
	One may readily check from our conditions on $\vf$ that $\sN_t-\sN_{(t-u)^+}$ is stochastically dominated 
	by a Poisson distribution of mean $u$ for each $t$, and it follows that the first term in~\eqref{bd2} vanishes as $t\to\infty$ for a.e.~$\circo,\o,\tio$; 
	\eqref{eqs5} then follows since $u$ is arbitrary and $\d_\ell$ is finite a.s.

\end{subsection}



\section{Environment seen from the particle}
\label{env}

\setcounter{equation}{0}

\noindent
We finally turn, in the last section of this paper, to the behavior of the environment seen from the particle {\em at jump times}. 
Our aim is to derive the convergence of its distribution as time/the number of jumps diverges, and to compare the limiting 
distribution with the product of invariant distributions of the marginal BD processes.
The main result of this section, stated next, addresses these issues under different subsets of the following set of conditions 
on the parameters of our process.
\begin{enumerate}
	\item\begin{equation}\label{cond1}\E(\xi_1)\ne0;\end{equation}
	\item\begin{equation}\label{cond2}\E(\|\xi_1\|^{2+\ve})<\infty; \end{equation}
	\item  \begin{equation}\label{cond3}\sx \mbox{ is transient, and } \pi \mbox{ has bounded support}; \end{equation}
	\item \begin{equation}\label{cond4}\inf_{n\geq0}\vf(n)>0,\end{equation}
\end{enumerate} 
out of which we compose the following conditions:
\begin{itemize}
	\item[$1'.$] Conditions~\eqref{cond1} and~\eqref{cond4} hold;
	\item[$2'.$] Conditions~\eqref{cond2} and~\eqref{cond4} hold;
	\item[$3'.$] Conditions~\eqref{cond3} and~\eqref{cond4} hold.
\end{itemize}
We note that in neither case we require monotonicity of $\vf$ \footnote{which is bounded above (by 1), as elsewhere in this paper}.

As anticipated, we focus on the homogeneously distributed case of the environment (i.e., we assume, as in the previous section, that
$p_n\equiv p\in(0,1/2)$) starting from a product of identical distributions on $\N$ with a positive exponential moment, 
and we will additionally assume that $\pi$ either has non zero mean, or has a larger than 2 moment.

Let $\o$ be a family of iid  homogeneous ergodic BD processes on $\N$ indexed by $\Zd$, starting from $\hat\mu_0$ as in the 
paragraph of~\eqref{expdec} of Section~\ref{ext}, and let $X$ be a time inhomogeneous random walk on $\Zd$ starting from 
$\zero$ in the environment $\o$, as in the prrevious sections. Let us recall that $\tau_n$ denotes the time of the $n$-th jump of $X$, 
$n\geq1$, and consider
\begin{equation}\label{envs}
	\vp_\x(n)=\o_{X(\tau_n-)+\x}(\tau_n),\,\x\in\Zd.
\end{equation}
$\vp(n):=\{\vp_\x(n),\,\x\in\Zd\}$ represents the environment seen by the particle {\em right before} its $n$-th jump.

\begin{teorema} \label{conv_env} \mbox{}
	Assume the condition stipulated on $\hat\mu_0$ in Section~\ref{ext} \footnote{see paragraph of~\eqref{expdec}}
	and suppose that $\E(\|\xi_1\|)<\infty$, and that, of the conditions listed above, at beginning of this section, either $1$, $2'$ or $3$ hold.
	Then 
	\begin{itemize}
		\item[$1.$] $\vp(n)$ converges in  $\P_{\hat{\mu}_{{0}}}$-distribution  (in the product topology on $\N^{\Zd}$) to $\vp:=\{\vp_\x,\,\x\in\Zd\}$,
		whose distribution does not depend on the particulars of initial distribution\footnotemark.
	\end{itemize}
	Moreover, if either $1'$, $2'$ or $3'$ hold, then
	\begin{itemize}
		\item[$2.$]  $\vp$ is absolutely continuous with respect to  $\hat\nu$.
	\end{itemize}
\end{teorema}
\footnotetext{i.e., it equals the one for the case where $\hat\mu_0=\hat\nu$}

\begin{observacao}\label{mono}
	There may be a way to adapt our approach in this section to relax/modify Condition 4 to some extent, by, say, requiring a slow decay of $\vf$ at $\infty$, perhaps adding monotonicity, as in the previous sections. But we do not feel that a full relaxation of that condition, even if imposing monotonicity, is within the present approach, at least not without substantial new ideas (to control tightness to a sufficient extent).
\end{observacao}

\begin{observacao}\label{short}
	As with results in previous sections, we do not expect our conditions for the above results to be close to optimal; again, our aim is to give reasonably natural conditions under which we are able to present an argument in a reasonably simple way. A glaring gap in our conditions is the mean zero $\xi_1$, not bounded away from zero $\vf$ case, even assuming monotonicity of $\vf$, as we did for the results in previous sections;  
	notice that the domination implied 
	by~(\ref{comp1},\ref{comp2}) holds for jump times of  $\brex$, not of $X$, and is thus not directly applicable, nor did we find an indirect application of it, or another way to obtain enough tightness for the environment at jump times of $X$ to get our argument going in that case.
\end{observacao}

As a preliminary for the proof of Theorem~\ref{conv_env}, we consider the (prolonged)  {\em backwards in time}  
random walk starting from $X(\tau_n-)$ (and moving backwards in time)
$\cx=\{\cx_\ell,\,\ell\in\N\}$ such  that $\cx_0=0$ and for $\ell\in\N_*$
\begin{equation}\label{back}
	\cx_\ell=\sum_{i=n-\ell}^{n-1}\xi'_i,
\end{equation}
where $\xi'_i=-\xi_i$, $i\geq1$, and we have prolonged $\xi$ to non positive integer indices in an iid way. 

Notice  that, for all $n\in\N_*$, $\cx$ is a random walk starting from $\zero$ whose (iid) jumps are distributed as $-\xi_1$
(and thus its distribution does not depend on $n$);
notice also  that $\cx_\ell=\sx_{n-1-\ell}-\sx_{n-1}$ for $0\leq\ell\leq n-1$.

It is indeed convenient to use a single backward random walk $\sz$ (with the same distribution as $\sy$) for all $n$. 
So in many arguments below, we condition on the trajectory of $\sz$ (which appears as a superscript in (conditional) probabilities below).

\medskip

\noindent{\em Proof of Theorem~\ref{conv_env}.}

We devote the remainder of this section for this proof.
Let $M$ be an arbitrary  positive integer, and consider the random vector 
\begin{equation}\label{vpm}
	\vp^M(n):=\{\vp_\x(n),\,\|\x\|\leq M\}.
\end{equation}
To establish the first assertion of Theorem~\ref{conv_env},  it is enough to show that $\vp^M(n)$ converges in distribution as $n\to\infty$.

We start by outlining a fairly straightforward  argument for the first assertion of  Theorem~\ref{conv_env} under Condition 3. In this case, we are again (as in the argument for the case of $d\geq3$ of Theorem~\ref {gclt} above), under the conditions for which we have cut times for the trajectory of $\sz$. 
It follows that there a.s.~exists a finite cut time $T_M$ such that the trajectory of $\sz$ after $T_M$ never visits $\{\x,\,\|\x\|\leq M\}$.
Then, assuming that the environment is started from $\bar\nu$, we have that, as soon as $n>T_M$, the conditional distribution of $\vp^M(n)$ given $\sz$ equals that of $\check\vp^{M,\sz}$, which is defined to be the  $\{\x,\,\|\x\|\leq M\}$-marginal of the environment of a process $(X,\o)$, with $\o$ started from $\hat\nu$ and $\sx$ started from $\sz_{T_M}$, seen at the time of the $T_M$-th jump of $X$ around the position it occupied immediately before that jump, with $\sx_\ell=\sz_{T_M-\ell}$, $0\leq\ell\leq T_M$. Notice that the result of the integration of the distribution of $\check\vp^{M,\sz}$ with respect to the distribution of $\sz$ does not depend on $n$; we may denote by $\check\vp^M$ the random vector having such (integrated) distribution. It is thus quite clear that $\vp^M(n)$ converges to $\check\vp^M$  in distribution as $n\to\infty$. That this also holds under the more general assumption on the initial environment stated in  Theorem~\ref{conv_env} can be readily argued via a coupling argument between $\hat\mu_0$ and $\hat\nu$, 
as done in Section~\ref{ext}. This concludes the proof of  Theorem~\ref{conv_env} under Condition 3.

Below, a similar argument, not however using cut times, will be outlined for the case where Condition 1 holds --- see Subsubsection~\ref{mune0}.

In order to obtain convergence of $\vp^M(n)$ when 
$\E(\xi_1)=0$, and either $d\leq2$ or $\pi$ has unbounded support, 
we require a bound on the tail of the  distribution of single-site marginal distributions of the environment at approriate jump times of $X$, to be specified below. (We also need Condition 2.)
In order to find such bound, we felt the need to further impose Condition 4. 
A bound of the same kind will also enter our argument for the second assertion of Theorem~\ref{conv_env}.
We devote the next subsection for obtaining this bound, and the two subsequent subsections for the conclusion of the proof Theorem~\ref{conv_env}.

\subsection{Bound on the tail of the marginal distribution of the environment}

\begin{lema}\label{benv} 
	Let $\x\in\Zd,\,m\geq1$ and suppose $\cR$ is a stopping time of $\sz$ such that $\cR\geq m$ a.s., and on $\{m\leq\cR<\infty\}$ we have that
	$\sz_\cR=\x$  and $\sz_{\cR-i}\ne\x$, $i=1,\ldots,m$. 
	Then, assuming that $\E(\xi_1)=0$, $\E(\|\xi_1\|^2)<\infty$ and that~\eqref{cond4} holds,  there exist a constant $\alpha>0$ and  $m_0\geq1$ 
	such that for all $m\geq m_0$, outside of an event involving $\sz$ alone of probability exponentially decaying in $m$, we have that  
	\begin{equation}\label{benv1} 
		\P^\sz_{\hat\mu_0}(\o_{\x-\sz_{n-1}}(\tau_{(n-\cR+\frac m2)^+})> (\log m)^2))\leq e^{-\alpha  (\log m)^2},
	\end{equation}
	for all $n\geq0$,  where $\P^\sz_{\hat\mu_0}$ denotes  the conditional probability $\P_{\hat\mu_0}(\cdot|\sz)$.
\end{lema}
\begin{proof}
	Given $n\geq0$, let us denote $\o_{\x-\sz_{n-1}}$ by $\o'_{\x}$.
	
	For $n\leq\cR$ 
	from Lemma~\ref{fbe}. 
	For this reason, we may restrict to the event where $n\geq\cR$.
	
	Let $\cR=\theta_0, \theta_1, \theta_2,\ldots$ be the successive visits of $\sz$ to $\x$ starting at $\cR$; 
	this may be a finite set of times; 
	set $I_0=0$,
	and, for $k>0$, set $I_k=\inf\{i>I_{k-1}:\theta_{i}-\theta_{i-1}> km\}$, $\cI_k=[\theta_{I_{k-1}},\ldots,\theta_{I_k-1}]\cap\Z$,
	$\cI'_k=(\theta_{I_{k}-1},\ldots,\theta_{I_k})\cap\Z$; 
	and let $\cI'=\cup_{k\geq1}\cI'_k$. Notice that  
	\begin{equation}\label{jk1}
		\sz\ne\x \text{ on }\cI', \text{ and }\, |\cI'_k|\geq km, \, k\geq1.
	\end{equation}
	Now let us consider $|\cI_k|$. We will bound the upper tail of its distribution. This will be based on a  bound to the upper tail of the distribution
	of $I_k-I_{k-1}$. A moment's thought reveals that, over all the cases of $\pi$ comprised in our assumptions, the worst case is the one dimensional, 
	recurrent case.
	For this case, and thus for all cases, Proposition 32.3 of~\cite{Spi76} yields that 
	\begin{equation}\label{ik1}
		\P\big(I_k-I_{k-1}>(km)^2\big)=\big[\P(\theta_1-\theta_0\leq km)\big]^{(km)^2}\leq \Big(1-\frac c{\sqrt{km}}\Big)^{(km)^2}\leq e^{-ckm}
	\end{equation}
	as soon as $m$ is large enough, where (here and below) $c$ is a positive real constant, not necesarily the same in each appearance. 
	(Here and below, we omit the subscript in the probability symbol when it may be restricted to the distribution of $\sz$ only.)
	We now note that $|\cI_k|\preceq km(I_k-I_{k-1})$,  and it follows that
	\begin{equation}\label{ik2}
		\P\big(|\cI_k|>(km)^{3}\big)\leq e^{-ckm}.
	\end{equation}
	
	It readily follows that, setting $\cJ=\cJ(\x,m)=\cap_{k\geq1}\big\{ |\cI_k|\leq(km)^{3}\big\}$, we have that
	\begin{equation}\label{k1}
		\P\big(\cJ^c\big)\leq e^{-cm},
	\end{equation}
	as soon as $m$ is large enough.
	
	Now, given $\sz$, let $K$ be such that $0\in\cI_K\cup\cI'_K$. We will assume that $\sz\in\cJ$ and $n\geq\cR$, and bound the conditional distribution of $\o'_{\x}(\tau_{n-\cR})$ given such $\sz$, via coupling, as follows.
	
	It is quite clear from the characteristics of $X$ that, given the boundedness assumption on $\vf$,  its jump times can be stochastically bounded from above and below by a exponential  random variables with rates 1 and $\d:=\inf\vf$, respectively, independent of $\o$.
	
	Let us for now 
	consider the succesive continuous time intervals $\mi_k,\mi_k'$, $1\leq k<K$ in the timeline of  $X$, 
	during which $\sz$ jumps in $\cI_k$, $\cI'_k$, $1\leq k\leq K$, respectively, if there is any such interval\footnote{$\mi_K'$ may be empty.}.
	We recall that time for $X$ and $\sz$ moves in different directions. 
	Let $\mt_k=|\mi_k|$, $\mt'_k=|\mi_k'|$ denote the respective interval lengths. 
	
	\begin{observacao}\label{frakbds}
		We note that, given our assumed bounds on $\vf$, whenever $1\leq k<K$, we have that $\mt'_k$ may be bounded from below by the sum of $km$ independent  standard exponential random variables. If $\sz\in\cJ$, then, for $1\leq k\leq K$, we have that $\mt_k$ may be bounded from above by the sum of $(km)^3$ iid exponential random variables of rate $\d$.
	\end{observacao}

	Whenever $K\geq2$, we introduce, enlarging the probability space if necessary, for each $1\leq k<K$, 
	versions of 
	$\o'_\x$ evolving at $\mi'_k$, namely $\oq_k$ and $\op_k$, coupled to $\o'_\x$ 
	so that, at $\tau_{n-\theta_{I_k}}$, $\oq_k$ is in equilibrium, and $\op_k$ equals the maximum of $\o$ in $\mi_{k+1}$, 
	and $\oq_k$, $\op_k$ and $\o'_\x$ evolve independently on $\mi'_k$ until any two of them meet, after which time they coalesce.
	Notice that it follows, in this case, that $\op_k\geq\o'_\x(\tau_{n-\theta_{I_k}})$.
	
	We now need an upper bound for the distribution of $\op_k$ at time $\tau_{n-\theta_{I_k}}$ assuming it starts from equilibrium  at time $\tau_{n-\theta_{I_k+1}}$. From the considerations on Remark~\ref{frakbds}, we readily find that it is bounded by the max of a BDP starting from
	equilibrium during a time of length given by the sum of $(km)^3$ iid exponential random variables of rate $\d$. In Appendix~\ref{frakbd}
	we give un upper bound for the tail of the latter random variable --- see Lemma~\ref{fbd} ---, which implies from the above reasoning that
	\begin{equation}\label{ub1}
		\sum_{w\geq0}\P^\sz\big(\op_k(\tau_{n-\theta_{I_k}})>(\log(km))^2|\o'_\x(\tau_{(n-\theta_{I_{k+1}-1})^+})=w\big)\nu(w)\leq e^{-c(\log(km))^2},
	\end{equation}
	$1\leq k<K$, where $\nu$, we recall, is the equilibrium distribution of the underlying environmental BDP. As follows from the proof of 
	Lemma~\ref{fbd} --- see Remark~\ref{gen_init} ---,~\eqref{ub1} also holds when we replace $\nu$ by any distribution on $\N$ with an 
	exponentially decaying tail; so, in particular, it holds for $k=K-1$ if we replace $\nu$ by the distribution of  
	$\o'_\x\big(\tau_{(n-\theta_{I_{K}-1})^+}\big)$, which may be checked to have such a tail --- see  Lemma~\ref{fbe}.  
	
	We now consider the events $A_k=\{\op_k(\tau_{n-\theta_{I_k}})\leq(\log(km))^2\}$, $k=1,\ldots, K$, and also the events
	$A'_k=\{$during $\mi'_k$, both $\oq_k$ and $\op_k$ visit the origin\}, $k=1,\dots,K-1$.
	
	\begin{observacao}\label{couple}
		In $A'_k$,  $\o'_\x$ and $\oq_k$ (and $\op_k$) coincide at time $\tau_{n-\theta_{I_k-1}}$. 
	\end{observacao}
	
	Given the drift of the BDP towards the origin, and the fact that $\cI'_k\geq km$, and also the lower bound on $\vf$, by a standard large deviation estimate, 
	we have that
	\begin{equation}\label{ub2}
		\P^\sz\big((A'_k)^c|A_k\big)\leq e^{-ckm},\,k=1,\ldots,K-1.
	\end{equation}
	
	Let us set $B_K=\cap_{k=1}^KA_k\cap\cap_{k=1}^{K-1}A'_k$, if $K\geq2$, and  $B_K\big|_{K=1}=A_1$. 
	From the reasoning in the latter two paragraphs above, 
	(given $\sz\in\cJ$, and minding Remark~\ref{couple}) we readily find that
	\begin{equation}\label{ub3}
		\P^\sz_{\hat\mu_0}\big((B_K)^c\big)\leq e^{-c(\log m)^2},\,K\geq2,
	\end{equation}
	and the same may be readily seen to hold also for $K=1$. Notice that the above bound is uniform in $K\geq1$.
	
	Combining this estimate with~\eqref{k1}, and from the fact that in $B_K$ we have that $\o'_{\x}(\tau_{n-\cR})\leq(\log m)^2$, and thus, given that 
	from $\cR-\frac m2$ to $\cR-1$, $\sz$ does not visit $\x$, and resorting again to a coupling of $\o'_\x$ to suitable $\oq_0$ and $\oq_0$ on the time interval $\big[\tau_{(n-\cR)^+},\tau_{(n-\cR+\frac m2)^+}\big]$, similarly as the ones of $\o'_\x$ to  $\oq_k$ and $\oq_k$ on $\mi'_k$, and to the lower bound on $\vf$, as well as to a standard large deviation estimate, as in the argument for~\eqref{ub2} above, the result  for the 
	case where $\o'_\x$ starts from equilibrium follows. If the initial distribution of $\o'_\x$ is not necessarily the equilibrium one, but 
	satisfies the conditions  of Section~\ref{ext}  --- see paragraph of~\eqref{expdec} ---, then, by the considerations at the end of the paragraph of~\eqref{ub1} above,  the ensuing arguments are readily seen to apply.

\end{proof}

\subsection{Conclusion of the proof of the first assertion of Theorem~\ref{conv_env}} 

\subsubsection{First case: $\E(\xi_1)=0$  and $d=1$}\label{m01}

\noindent
In this case we also assume, according to the conditions of Theorem~\ref{conv_env} , that $\E(|\xi_1|^{2+\ve})<\infty$,
for some $\ve>0$; we may assume, for simplicity, that $\ve\leq2$.

Given $\sz$ and $M\in\N$, consider the following (discrete) stopping times of $\sz$: 
$\fvt_1=\inf\{n>0: |\sz_n|>\Upsilon_1\}$, with $\Upsilon_1=M$, and for $\ell\geq1$,
\begin{eqnarray}\label{vt1}
	\bvt_{\ell}&=& 
	\begin{cases}
		\inf\big\{n>\fvt_{\ell}: \sz_n<\sz_{\fvt_{\ell}}\big\},&\text{ if } \sz_{\fvt_{\ell}}>0;\\
		\inf\big\{n>\fvt_{\ell}: \sz_n>\sz_{\fvt_{\ell}}\big\},&\text{ if } \sz_{\fvt_{\ell}}<0,
	\end{cases}\\\label{vt2}
	\fvt_{\ell+1}&=& 	\inf\big\{n\geq\bvt_{\ell}: |\sz_n|>\Upsilon_{\ell+1}\big\},
\end{eqnarray}
where $\Upsilon_{\ell+1}=\max\big\{|\sz_{n}|,\,n<\bvt_{\ell}\big\}$.

For $\ell\geq1$, $\fvt_\ell$ indicates the times where $|\sz|$ exceeds the previous maximum (above $M$), say at respective values $x_\ell$, and $\bvt_\ell$ the return time after that to either a value below $x_\ell$, if $x_\ell>0$, or a value above $-x_\ell$, if $x_\ell<0$. Notice that we may have 
$\fvt_{\ell+1}=\bvt_{\ell}$ for some $\ell$ 
\footnote{The first moment condition on $\xi_1$ can however be readily shown to imply that the set of such $\ell$ is a.s.~finite; this remark is however not taken advantage of in the sequel.}.
Figure~\ref{fig:cross} illustrates realizations of these random variables.

\begin{figure}[!htb]
	\centering
	\includegraphics[scale=0.4,trim=0 55 0 0,clip]{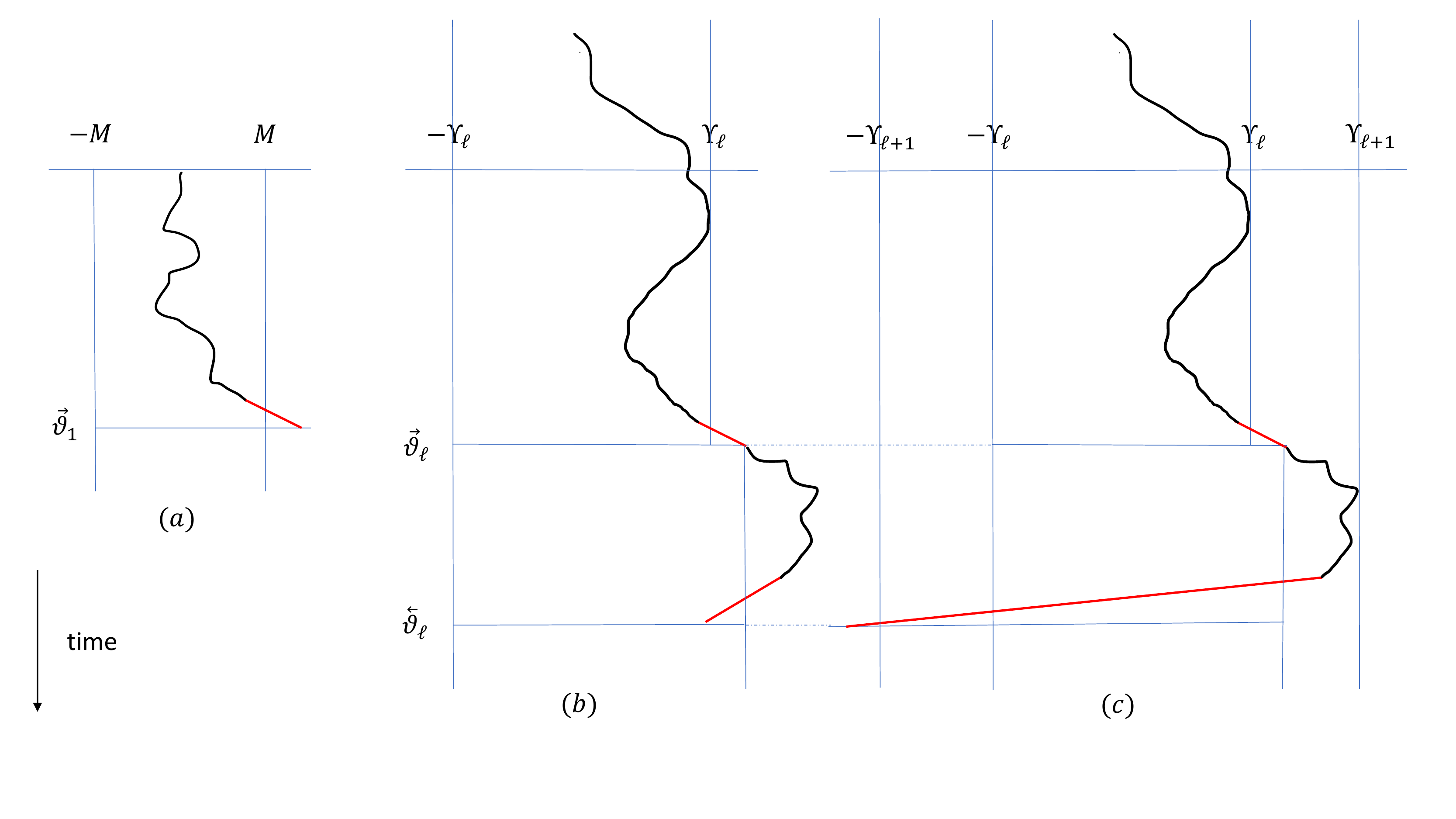}
	\caption{Illustration of occurrences of random variables introduced in~(\ref{vt1},\ref{vt2}). 
		In (b) and (c), edges in red represent single jumps.
		In (c)  we have  $\fvt_{\ell+1}=\bvt_{\ell}$, but not in (b).}
	\label{fig:cross} 
\end{figure}

Now set, for $\ell\geq1$
\begin{equation}\label{vt3}
	\vr_\ell=\bvt_{\ell}-\fvt_{\ell}\quad\text{and}\quad \chi_\ell=|\sz_{\fvt_\ell}|-\Upsilon_\ell.
\end{equation}

We will argue in Appendix~\ref{rwrs} 
that for $m\geq1$
\begin{eqnarray}\label{u1}
	&\sup_{\ell}\P(\vr_\ell>m)\geq\frac{\text{const}}{\sqrt m};&\\\label{u2}
	&\sup_{\ell}\P(\chi_\ell>m)\leq\frac{\text{const}}{m^\ve}.&
\end{eqnarray}

For $m\geq1$ fixed, let $\L_m=\inf\{\ell\geq1:\vr_\ell\geq m\}$. It readily follows from~\eqref{u1} that
\begin{equation}\label{vt4}
	\P(\L_m>m)\leq\Big(1-\frac{\text{const}}{\sqrt m}\Big)^m\leq e^{-c\sqrt m},
\end{equation}
for some constant $c>0$, and it follows from~\eqref{u2} that for $b>0$
\begin{equation}\label{vt5}
	\P\big(\max_{1\leq\ell\leq m}\chi_\ell>m^b\big)\leq \text{const } m^{1-b\ve},
\end{equation}
so that the latter probability vanishes as $m\to\infty$ for $b>1/\ve$.

We now notice that, pointing out to Appendix~\ref{rwrs}  for the definitions of $\bt^-_0$ and  $\bt^+_0$, that\break  
$\big(\chi'_\ell:=\Ups_{\ell+1}-|\sz_{\fvt_\ell}|,\vr_\ell\big)$ is distributed as 
$\big(\!\max_{0\leq i<\bt^-_0}\sz_i,\bt^-_0\big)$, if $\sz_{\fvt_\ell}>0$, 
and as $\big(\!-\min_{0\leq i<\bt^+_0}\sz_i,\bt^+_0\big)$, if $\sz_{\fvt_\ell}<0$. In the former case, we have for $k\geq1$
\begin{equation}\label{vt6}
	\P(\chi'_\ell>k;\,\vr_\ell\leq m)\leq \P\big(\max_{0\leq i\leq m}\sz_i>k\big)\leq\text{const }\frac{m}{k^2},
\end{equation}
where the latter passage follows from Kolmogorov's Maximal Inequality, and the same bound holds similarly in the latter case.

We thus have that, for $b>0$,
\begin{equation}\label{vt7}
	\P\big(\max_{1\leq\ell<\L_m}\chi'_\ell>m^b;\,\L_m\leq m\big)\leq\text{const }m^{1-2b}\leq\text{const }m^{1-b\ve},
\end{equation}
and, since $\Ups_{\ell}=M+\sum_{i=1}^{\ell-1}(\chi_i+\chi'_i)$, it follows that 
\begin{equation}\label{vt8}
	\P\big(\Ups_{\L_m}>m^{1+b}\big)\leq\text{const }m^{1-b\ve},
\end{equation}
which thus vanishes as $m\to\infty$ as soon as $b>1/\ve$.

\smallskip

We may now proceed directly to showing that 
\begin{equation}\label{cau1}
	\{\vp^M(n),\,n\geq1\}\,\text{ is a Cauchy sequence in distribution.}
\end{equation}

Let ux fix $m$ as above and $\x\in\{-m^{b+1},\dots,m^{b+1}\}$; we assume that $m>M^{1/(b+1)}$; we next set 
\begin{equation}\label{R}
	\cR=\inf\{n\geq\bvt_{\L_m}:\,\sz_n=\x\},
\end{equation}
which satisfies the conditions of Lemma~\ref{benv}, and we thus conclude from that lemma  
that, if $\sz\in\cJ$, then
\begin{equation}\label{benv2} 
	\limsup_{n\to\infty}\P^\sz_{\hat\mu_0}\big(\o'_{\x}(\tau_{(n-\cR+\frac m2)^+})>(\log m)^2\big)\leq e^{-\alpha  (\log m)^2},
\end{equation}
and thus that
\begin{equation}\label{benv3} 
	\limsup_{n\to\infty}\P^\sz_{\hat\mu_0}\Big(\max_{\x\in\{-m^{b+1},\dots,m^{b+1}\}}\o'_{\x}(\tau_{(n-\cR+\frac m2)^+})>(\log m)^2\Big)\leq
	e^{-c  (\log m)^2},
\end{equation}
where $c$ is a positive number, depending on $\alpha$ and $b$ only. It readily follows from the argumemts in the proof of Lemma~\ref{benv} 
(namely, the coupling of $\o'_\x$ and $\oq$) that
\begin{equation}\label{benv4} 
	\limsup_{n\to\infty}\P^\sz_{\hat\mu_0}\Big(\max_{\x\in\{-m^{b+1},\dots,m^{b+1}\}}\o'_{\x}(\tau_{(n-\bvt_{\L_m}+\frac m2)^+})>(\log m)^2\Big)\leq
	e^{-c  (\log m)^2}.
\end{equation}

Now let $\sz\in\cJ\cap\{\Ups_{\L_m}\leq m^{1+b}\}$, and choose $n_0\geq \bvt_{\L_m}$ so large that for $n\geq n_0$ we have
\begin{equation}\label{benv5} 
	\P^\sz_{\hat\mu_0}\Big(\max_{\x\in\{-\Ups_{\L_m},\dots,\Ups_{\L_m}\}}\o'_{\x}(\tau_{n-\bvt_{\L_m}+\frac m2})>(\log m)^2\Big)\leq
	e^{-c  (\log m)^2}.
\end{equation}
\begin{figure}[!htb]
	\centering
	\includegraphics[scale=0.48,trim=80 30 180 70,clip]{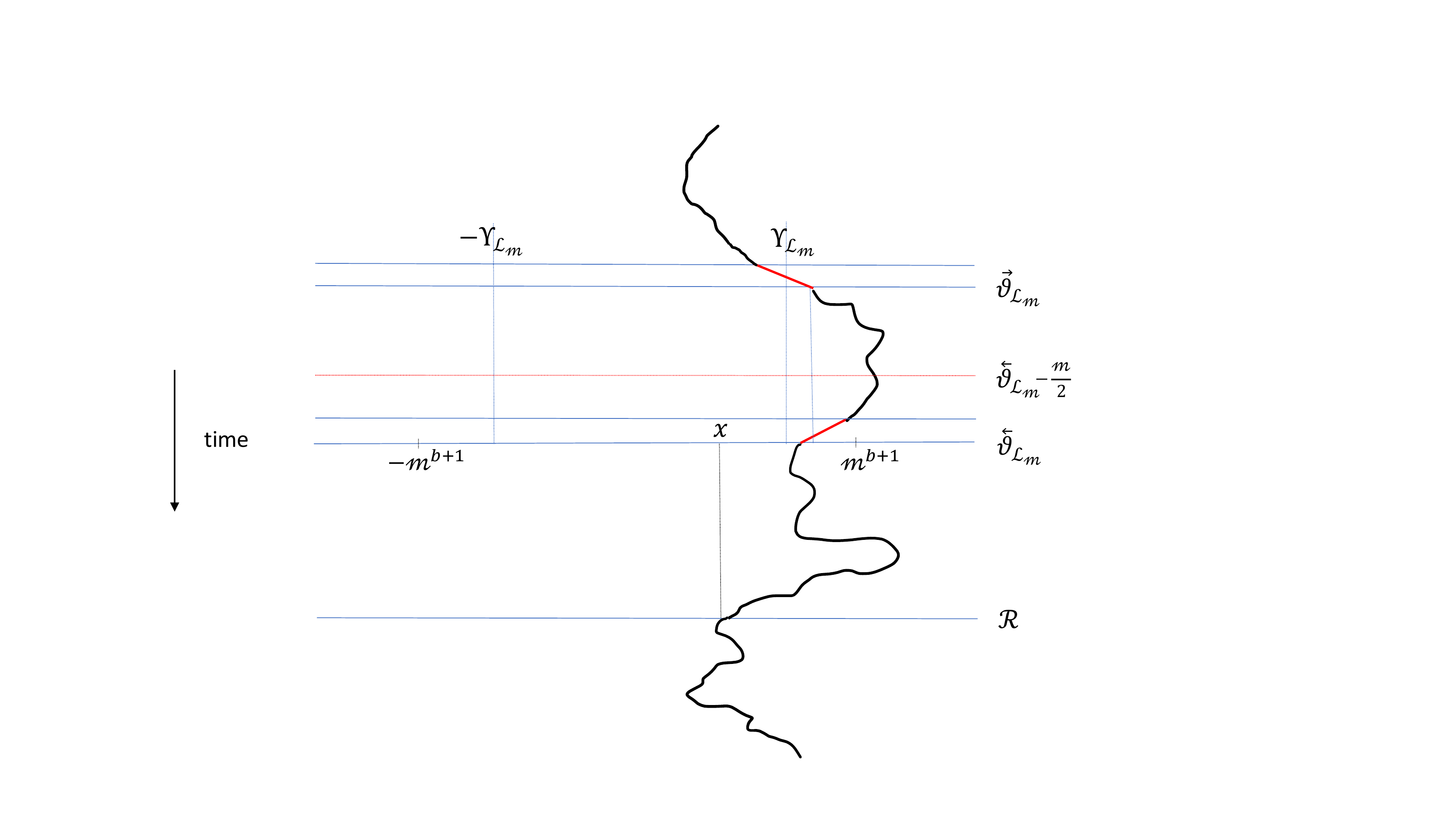}
	\caption{Schematic depiction of a stretch of the (backward) trajectory of $\sz\in\{\Ups_{\L_m}\leq m^{1+b}\}$.}
	\label{fig:ilu42} 
\end{figure}

Now let us consider a family of BDPs on the timelines of $\x\in\{-\Ups_{\L_m},\dots,\Ups_{\L_m}\}-\sz_{n-1}$,  which we
denote by $(\oq_\x)_{\x\in\{-\Ups_{\L_m},\dots,\Ups_{\L_m}\}-\sz_{n-1}}$, 
starting at time $\tau_{n-\bvt_{\L_m}+\frac m2}$ in the product of equilibrium distributions
$\hat\nu_m:=\!\!\!\bigotimes\limits_{\x\in\{-\Ups_{\L_m},\dots,\Ups_{\L_m}\}-\sz_{n-1}}\!\!\!\ \nu$, 
independently of $\o\big(\tau_{n-\bvt_{\L_m}+\frac m2}\big)$, 
with $\oq_\x$ coupled to $\o_\x$ so that they move independently till first meeting, after which time they coalesce. 

\begin{observacao}\label{eq}
	We have that $\big(\oq_\x(\tau_{n-\fvt_{\L_m}+1})\big)_{\x\in\{-\Ups_{\L_m},\dots,\Ups_{\L_m}\}-\sz_{n-1}}\sim\hat\nu_m$.
	This follows from the fact that in the period from $\tau_{n-\bvt_{\L_m}+\frac m2}$ to $\tau_{n-\fvt_{\L_m}+1}$ the jump time lengths
	of $X$ depend solely on $\big(\o_\x(\tau_{n-\bvt_{\L_m}+\frac m2})\big)_{\x\notin\{-\Ups_{\L_m},\dots,\Ups_{\L_m}\}-\sz_{n-1}}$, 
	and on the birth-and-death processes 
	evolving on timelines of $\Z\setminus\{-\Ups_{\L_m},\dots,\Ups_{\L_m}\}-\sz_{n-1}$.
	
\end{observacao}

Arguing similarly as in the proof of Lemma~\ref{benv} --- see Remark~\ref{couple} ---, we find that
\begin{equation}\label{benv6} 
	\P^\sz_{\hat\mu_0}\Big(\o_{\x}(\tau_{n-\fvt_{\L_m}+1})\ne\oq_{\x}(\tau_{n-\fvt_{\L_m}+1})\,
	\text{for some }\x\in\{-\Ups_{\L_m},\dots,\Ups_{\L_m}\}-\sz_{n-1}\Big)\leq
	e^{-c  m}.
\end{equation}

Now, given $\sz\in\cJ\cap\{\Ups_{\L_m}\leq m^{1+b}\}$, let $\vpq(n):=(\vpq_\x(n))_{\x\in\Z}$ represent the environment at time $n$ of a time 
inhomogeneous Markov jump process $(X(t),\o(t))$ starting at time $\tau_{n-\fvt_{\L_m}}$ from $\hat\nu$
(with $X$ starting at that time from $\sz_{\fvt_{\L_m}-1}$, and jumping,  forwards in time, along the backward trajectory of $\sz$).

\begin{observacao}\label{eq2}
	Notice that the distribution of $\vpq(n)$ does not depend on $n$.
\end{observacao}

Since, given $\sz$, the distribution of $\vp^M(n)$ depends only on environments at timelines of sites in 
$\{-\Ups_{\L_m},\dots,\Ups_{\L_m}\}-\sz_{n-1}$
from time $\tau_{n-\fvt_{\L_m}}$ to $\tau_n$, we have, for $\sz\in\cJ\cap\{\Ups_{\L_m}\leq m^{1+b}\}$, and in the complement
of the event under the probability sign on the left hand side of~\eqref{benv6}, and resorting to an obvious coupling, that 
$\vp^M_\x(n)=\vpq_\x(n)$ for $|x|\leq M$.

Now, finally, given $\ve>0$, we may choose $m$ large enough, and then $n_0$ large enough so that, for 
$\sz\in\cJ\cap\{\Ups_{\L_m}\leq m^{1+b}\}$, and by~\eqref{benv6}, we have that the distance\footnote{associated  to the usual product topology}
of the conditional distribution of $\vp^M(n)$ given $\sz$ to that of 
the conditional distribution of $\big(\vpq_\x(n)\big)_{|\x|\leq M}$ given $\sz$ is smaller than $\ve$.
We conclude from Remark~\ref{eq2} that the sequence in $n$ of conditional distributions of $\vp^M(n)$ given $\sz$ is Cauchy, and 
thus the same holds for the sequence of unconditional distributions. It readily follows from the above arguments that the limit 
is the same no matter what are the details of $\hat\nu_0$ satisfying the conditions in the paragraph of~\eqref{expdec}.

\subsubsection{Second case: $\E(\xi_1)=0$  and $d\geq2$}\label{m02+}


\noindent
Let us first note that each coordinate of $\sz$ performs mean zero walks with the same moment condition as in the $d=1$ case, 
so arguments for that case apply to, say, the first coordinate, and we get control of the location of that coordinate at (backward)
time $\bvt_{\L_m}$: it is with high probability in $\{-m^{1+b},\ldots,m^{1+b}\}$ if $m$ is large, according to~\eqref{vt8}.
But now we need to control the location of the other coordinates; and we naturally seek  a similar polynomial such control
as for the first coordinate.

In order to achieve that, we will simply show that, with high probability, we have polynomial control on the size of $\bvt_{\L_m}$,
and that follows from standard arguments once we condition on the  event that $\Ups_{\L_m}\leq m^{1+b}$ --- which
has high probability according to~\eqref{vt8}. Indeed, on that event $\bvt_{\L_m}$ is stochastically dominated by 
$\vartheta^*=\inf\{n>0: |\sz_n|>m^{1+b}\}$, the hitting time by $\sz$ of the complement of $\{-m^{1+b},\ldots,m^{1+b}\}$.
It is well known that under our conditions we have that $\vartheta^*\leq m^{2(1+b+\d)}$ with high probability for all $\d>0$
--- see Theorem 23.2 in~\cite{Spi76}, and thus, again recalling a well known result (see Theorem 23.3 in~\cite{Spi76}), we have that
the max over the $j=$ 2 to $d$, and times from 0 to $\bvt_{\L_m}$, of the absolute value of $j$-th coordinate of $\sz$ is bounded
by $m^{1+b+\d}$ with high probability for any $\d>0$.

With this control over the maximum dislocation of $|z|$ from time 0 to $\bvt_{\L_m}$, we may repeat essentially the same argument
as for $d=1$ (with minor and obvious changes).

\subsubsection{Last case: $\E(\xi_1)\ne0$}\label{mune0}


\noindent
A similar, but simpler approach works in this case as in the previous case. 

Let us assume without loss that $\E(\xi_1(1))<0$ --- so that $\E(\xi'_1(1))>0$, where the '1' within parentheses indicate the coordinate. 
We consider the quantities introduced in Subsubsection~\ref{m01} for $\sz(1)$ instead of $\sz$,
and let $\L_\infty=\inf\{\ell\geq1:\vr_\ell=\infty\}$. It is quite clear that this is an a.s.~finite random variable. 
Now, given a typical $\sz$, as soon as $n\geq\L_\infty$, we have that $\vp^M(n)$ is again distributed as $\vpq(n)$ given above; see paragraph 
right below~\eqref{benv6}; by Remark~\ref{eq2}, this distribution does not depend on $n$; notice that the latter definition and property make 
sense and hold in higher dimensions as well. 
The result follows for $\vp^M(n)$ conditioned on $\sz$, and thus also for the unconditional distribution. 

Notice that we did not need a positive lower bound for $\vf$ (and neither a finite uper bound).

\begin{observacao}\label{cond}
	We note that the above proof established, in every case, the convergence of the conditional distribution of $\vp(n)$ given $\sz$ as $n\to\infty$ 
	to, say, $\vp^\sz$.
\end{observacao}

\begin{observacao}
	It is natural to ask about the asymptotic environment seen by the particle at large deterministic times. 
	A strategy based on looking at the environment seen at the most recent jump time, 
	which might perhaps allow for an approach like the above one, 
	seems to run into a sampling paradox-type issue,  which may pose considerable difficulties in the inhomogeneous setting.
	We chose not to pursue the matter here.
\end{observacao}

\subsection{Proof of the second assertion of Theorem~\ref{conv_env}}
It is enough, taking into account Remark~\ref{cond}, to show the result for the limit of the conditional distribution of $\vp(n)$ given $\sz$, which we denote by $\vp^\sz$, for $\sz$ in an event of arbitrarily large probability, as follows. 

For $N\geq0$, let $\fQ_N=\{\x\in\Z^d:\,\|\x\|\leq N\}$ and $T_N=\inf\{k\geq0:\,\|\sz_k\|\geq N\}$.
One may check that, by our conditions on the tail of $\pi$ and the Law of Large Numbers, we have that for some $a>0$ 
there a.s.~exists $N_0$ such that for all $N\geq N_0$ we thave that $T_N>aN$.
For $\x\in\Zd$, let $\cR_\x=\inf\{k\geq0:\,\sz_k=\x\}$, and let $\fR=\{\x\in\Zd:\,\cR_\x<\infty\}$.

Consider now the event $\tilde\cJ_N:=\cap_{\x\in\fR,\|\x\|\geq N}\cJ(\x,a\|\x\|)$, with $\cJ(\cdot,\cdot)$ as in the paragraph of~\eqref{k1}.
It follows from~\eqref{k1} that
\begin{equation}\label{jm}
	\P\big(\tilde\cJ_N^c\big)\leq e^{-cN},
\end{equation}
for some positive constant $c$ (again, not the same in every appearance).
Lemma~\ref{benv} and the remark in the above paragraph then ensure that  for $\x\in\fR$ such that
$\|\x\|\geq N\geq a^{-1}m_0\vee N_0$ and a.e.~$\sz\in\tilde\cJ_N$, we have that
\begin{equation}\label{abc1} 
	\P^\sz_{\hat\mu_0}\big(\o_{\x-\sz_{n-1}}(\tau_{(n-\cR_\x+\frac{a\|\x\|}2)^+})>(\log a\|\x\|)^2\big)\leq e^{-c( (\log a\|\x\|)^2)};
\end{equation}
it readily follows fromm Lemma~\ref{fbe} that the same bound holds for $\x\notin\fR$;

For $\|\x\|\geq N\geq a^{-1}m_0\vee N_0$, let us couple $\o_\x$ from $\tau_{(n-\cR_\x+\frac{a\|\x\|}2)^+}$ onwards, in a coalescing way, 
as done multiple times above, to $\oq_\x$, a BDP starting at $\tau_{(n-\cR_\x+\frac{a\|\x\|}2)^+}$ from $\nu$, its equilibrium distribution.
We assume that $\oq_\x$, $\|\x\|\geq N$, are independent. 
It readily follows, from arguments already used above, that, setting 
$\fC_N=\cap_{\|\x\|\geq N}\{\o_{\x-\sz_{n-1}}(\tau_n)=\o_{\x-\sz_{n-1}}(\tau_n)\}$,
we have that 
\begin{equation}\label{abc2} 
	\P^\sz_{\hat\mu_0}\big(\fC_N^c\big)\leq e^{-cN}. 
\end{equation}

For $N\geq0$ and $n>T_N$, let us consider $\ozq(n)=(\hozq(n),\chozq)$,  where $\hozq(n)$ is $\vp^\sz(n)$ restricted to $\fQ_N$, 
and $\chozq$ is distributed as the product of $\nu$ over  $\N^{\Zd\setminus\fQ_N}$, independently of $\hozq(n)$. 
The considerations of the previous paragraph imply that we may couple $\vp^\sz(n)$ and $\ozq(n)$ so that they coincide outside a probability 
which vanishes as $N\to\infty$ {\em uniformly} in $n>T_N$. 

Now set $\ozq:=(\hozq,\chozq)$,  where $\hozq$ is $\vp^\sz$ restricted to $\fQ_N$, 
and $\chozq$ is distributed as the product of $\nu$ over  $\N^{\Zd\setminus\fQ_N}$, independently of $\hozq$. 
From the first item of Theorem~\ref{conv_env}, we have that $\hozq(n)$ converges in distribution to $\hozq$ as $n\to\infty$.
It follows from this and the considerations in the previous paragraph that we may couple $\vp^\sz=(\hozq,\choz)$ and $\ozq$ so that
$\choz=\chozq$ outside an event of vanishing probability as $N\to\infty$.

In order to conclude, let us take an event $A$ of the product $\sigma$-algebra generated by the cylinders of $\N^{\Zd}$ such that 
$\hat\nu(A)=0$.  Given $\eta\in\N^{\fQ_N}$, let $A_\eta=\{\zeta\in\N^{\Zd\setminus\fQ_N} : (\eta,\zeta)\in A\}$.
Since $\hat\nu$ assigns positive probability to every cylindrical configuration of $\N^{\Zd}$, it follows that $\check\nu(A_\eta)=0$ 
for every $\eta\in\N^{\fQ_N}$, where $\check\nu$ is the product of $\nu$'s over $\N^{\Zd\setminus\fQ_N}$.
Thus, resorting to the coupling in the previous paragraph, we find that
\begin{equation}\label{abc3}
	\P^\sz_{\hat\mu_0}\big(\vp^\sz\in A\big)\leq \P^\sz_{\hat\mu_0}\big(\ozq\in A\big)+\P^\sz_{\hat\mu_0}\big(\choz\ne\chozq\big)
	=\P^\sz_{\hat\mu_0}\big(\choz\ne\chozq\big),
\end{equation}
since the first probability on the right hand side equals
\begin{equation}\label{abc4}
	\sum_{\eta\in\N^{\fQ_N}}\check\nu\big(A_\eta\big)\P^\sz_{\hat\mu_0}\big(\hat\vp^\sz=\eta\big)=0,
\end{equation}
as follows from what was pointed out above in this paragraph. Since the right hand side of~\eqref{abc3} vanishes as $N\to\infty$, 
as pointed out at the end of the one but previous paragraph, and the left hand side does not depend on $N$, we have that 
$\P^\sz_{\hat\mu_0}\big(\vp^\sz\in A\big)=0$, as we set out to prove.



\appendix

\section{Equivalence of~\eqref{extracon} and~\eqref{extracor}}
\label{app}

\setcounter{equation}{0}

The formulas we will present below may be found in the literature, either explicitly or from other explicit formulas~\cite{KMcG}. For this reason, but also in an attempt of self-containment, we will deduce them, rather briefly and sketchily, trusting the reader to be readily able to fill details in (or go to the literature).

For $n\in\N_*$, let $\T_n$ and $\ss_n$ denote $\Er_{n}(T_{n-1})$ and $\Er_{n}(T^2_{n-1})$, respectively. Under $\Pr_{n}$, by the Markov property, we have that
\begin{equation}\label{rec}
	T_{n-1}=1+\oone\{\w_1=n+1\}(T_n'+T_n''),
\end{equation}
where $\oone\{\w_1=n+1\},T_n',T_n''$ are independent, $T_n'$ has the same distribution as $T_n$ and $T_n''$  has the same distribution as $T_{n+1}$ under 
$\Pr_{n+1}$. It follows that for $n\geq1$
\begin{equation}\label{rec1}
	\T_n=\frac1{q_n}+\rho_n\T_{n+1}\quad\mbox{and}\quad\ss_n=\s_n+\rho_n\ss_{n+1},
\end{equation}
where $\s_n=s_n/q_n$, and $s_n=1+2p_n(\T_n+\T_{n+1}+\T_n\T_{n+1})$.\footnote{The finitude of $\T_n$ follows from the ergodicity of $\w$; that of $\ss_n$ may be checked by comparison to versions of $\w$ reflected at large states, using a similar reasoning as in the present argument for the reflected versions of $\ss_n$, which are more obviously finite.}
It then readily follows that
\begin{equation}\label{rec2}
	\T_n=\frac{1}{R_{n-1}}\sum_{\ell\geq n}\frac1{q_\ell}R_{\ell-1}\quad\mbox{and}\quad\ss_1=\sum_{\ell\geq 1}{\s_\ell}R_{\ell-1}.
\end{equation}
Our condiitons on $\p,\q$ imply that $\T_n\asymp\frac{S_{n-1}}{R_{n-1}}$ and $\s_n\asymp\T_n\T_{n+1}$. It follows that
\begin{equation}\label{rec3}
	\ss_1\asymp\sum_{\ell\geq 1}\frac{S_{\ell-1}S_\ell}{R_\ell}\asymp\sum_{\ell\geq 1}\frac{S^2_\ell}{R_\ell},
\end{equation}
and the second equivalence is established. For the first one, we have that $\Er_\nu(T_0)<\infty$ if and only if
\begin{equation}\label{rec4}
	\infty>\sum_{n\geq0}R_{n}\sum_{i=0}^{n}\frac{S_i}{R_i}=\sum_{i\geq 0}\frac{S^2_i}{R_i},
\end{equation}
completing the argument.



\section{Justification of~\eqref{u1} and~\eqref{u2}}
\label{rwrs}

\setcounter{equation}{0}

Notice that the distribution of $\vr_\ell$ depends on $\ell$ only through the sign of $\cx_{\fvt_{\ell-1}}$ and equals that of
$\bt^-_0$, if $\cx_{\fvt_{\ell-1}}>0$, and that of $\bt^+_0$, if $\cx_{\fvt_{\ell-1}}<0$, 
where for $L\geq0$
\begin{equation}\label{tl}
	\bt^-_L:=\inf\{n\geq1:\,\sy_n<-L\}\,\text{ and }\,\bt^+_L:=\inf\{n\geq1:\,\sy_n>L\}.
\end{equation}

Also, given $\fvt_{\ell-1}=J$ and $\cx_{\fvt_{\ell-1}}=K>0$, we have that for $m\geq1$
\begin{equation}\label{tl1}
	\{\chi_\ell>m\}\subset
	\big\{\sy_{\bvt_{\ell-1}}-K>m\big\}\cup
	\left\{\big\{\sy_{\bvt_{\ell-1}}\in(-K,K)\big\}\cap\Big\{\big\{\chi'_\ell>m\big\}\cup\big\{\chi''_\ell>m\big\}\Big\}\right\},
\end{equation}
where $\chi'_\ell=\sy_{\vr'_\ell}-\Upsilon_\ell$, with $\vr'_\ell=\inf\{n>\bvt_{\ell-1}:\,\sy_n>\Upsilon_\ell\}$, and
$\chi''_\ell=|\sy_{\vr''_\ell}|-\Upsilon_\ell$, with $\vr''_\ell=\inf\{n>\bvt_{\ell-1}:\,\sy_n<-\Upsilon_\ell\}$.

A similar reasoning holds in the case where $K<0$, and it readily follows 
that, in order to justify~\eqref{u1} and~\eqref{u2}, it is enough to show that
\begin{eqnarray}\label{u3}
	&\P(\bt^+_0>m),\,\P(\bt^-_0>m)\geq\frac{\text{const}}{\sqrt m};&\\\label{u4}
	&\sup_{L\geq0}\P\big(\sy_{\bt^+_L}-L>m\big),\,\sup_{L\geq0}\P\big(|\sy_{\bt^-_L}|-L>m\big)\leq\frac{\text{const}}{m^\ve}.&
\end{eqnarray}

(\ref{u4}) follows from the known, stronger result that both expressions in it are asymptotic to const$/\sqrt{m}$ as $m\to\infty$
--- see e.g.~\cite{Em}, Theorem 1, and the remark which follows its statement.

Having not found a direct reference for~(\ref{u4}), which we suspect may exist, we develop 
a strategy based on material of~\cite{Spi76}, wherefrom that claim may be argued, as follows.

Our argument 
goes through an adaptation of a hint to an exercise proposed in 
Chapter IV of~\cite{Spi76}, namely, Exercise 6 in that chapter.
The hint, describing an approach attributed by the author to Harry Kesten (without further reference) to give a bound on the moments of what in our notation is $\sy_{\bt^+_L}$, can be adapted to yield the following. 
Let us argue the former case of~\eqref{u4}; the latter case is similar.


Following Kesten's approach, as outlined in the above mentioned hint, we may write
\begin{eqnarray}
	\P\big(\sy_{\bt^+_L}-L>m\big)&=&\sum_{k>m}\sum_{z\leq L}\,g_{(L,\infty)}(0,z)\P(\xi'_1=L-z+k)\nn\\\label{hk1}
	&=&\sum_{z\leq L}\,g_{(L,\infty)}(0,z)\P(\xi'_1>L-z+m),
\end{eqnarray}
where $g_{(L,\infty)}(0,z)$ is the expected number of visits of $\sy$ to $z$ before going above $L$ for the first time.

We claim  that 
\begin{equation}\label{hk2}
	g_{(L,\infty)}(0,z)\leq\text{const }\{(L-z+1)\wedge (L+1)\},\,z\leq L. 
\end{equation}

It follows that $\P\big(\sy_{\bt^+_L}-L>m\big)$ is bounded above by constant times
\begin{equation}\label{hk3}
	(L+1)\sum_{w\geq L}\P(\xi'_1>w+m)+\sum_{w\geq 0}(w+1)\P(\xi'_1>w+m),
\end{equation}
and the $2+\ve$ moment condition on $\xi_1$ readily implies the former case of~\eqref{u4}.

To justify~\eqref{hk2}, we first notice that, by a standard use of the Markov property, we have that, for $0<z\leq L$,
\begin{equation}\label{h1}
	g_{(L,\infty)}(0,z)\leq\sup_{w\leq0}g_{(L-z+1,\infty)}(0,w),
\end{equation}
and for $z\leq0$
\begin{equation}\label{h2}
	g_{(L,\infty)}(0,z)\leq g_{(0,\infty)}(0,z)+\sup_{w<0}g_{(L-1,\infty)}(0,w).
\end{equation}

Setting $\G_k=\sup_{w\leq0}g_{(k,\infty)}(0,w)$, $k\geq0$, we find from~(\ref{h1},\ref{h2}) that 
\begin{equation}\label{h3}
	\G_L\leq (L+1)\G_0\,\text{ and, for } 0<z\leq L, \, g_{(L,\infty)}(0,z)\leq (L-z+1)\G_0.
\end{equation}

It remains thus to show that
\begin{equation}\label{h4}
	\G_0<\infty.
\end{equation}

In order to do that, we proceed as follows. Let $\Theta_0=0$ and, for $n\geq1$, $\Theta_n=\inf\{\ell>\Theta_{n-1}:\,\sy_\ell=0\}$.
$(\Theta_n)_{n\geq0}$ are the successive passages of $\sy$ by the origin. They are a.s.~all finite under our conditions. 
Now for $n\geq1$ and $z<0$, let $N_n(z)=\sum_{i=\Theta_{n-1}}^{\Theta_n}\oone\{\sy_i=z\}$ be the number of visits of $\sy$ to $z$ 
between times $\Theta_{n-1}$ and $\Theta_n$. Finally, set $\tau^+=\inf\{n\geq0:\,\sy_{\Theta_n+1}>0\}$.

We then have for $z<0$ that
\begin{equation}\label{h5}
	g_{(0,\infty)}(0,z)\leq\E_0\Big\{\sum_{n=1}^{\tau^+}N_n(z)\Big\},
\end{equation}
where $\sum_{n=1}^0\cdots=0$ by convention. One can readily check that the latter expectation equals
\begin{equation}\label{h6}
	\E_0(\tau^+)\,\E_0\big(N_1(z)|\xi'_1<0\big).
\end{equation}

Now, from the identity
\begin{equation}\label{h7}
	\E_0\big(N_1(z)|\xi'_1<0\big)\P_0(\xi'_1<0)+\E_0\big(N_1(z)|\xi'_1>0\big)\P_0(\xi'_1>0)=\E_0\big(N_1(z)\big)=1,
\end{equation}
where the last equality follows from an application of Proposition 3 in Chapter 3 of~\cite{Spi76}, Section 11, we find that
the latter expectation in~\eqref{h6}  is bounded above by  $1/\P_0(\xi'_1<0)<\infty$.

The  former expectation in~\eqref{h6} is quite clearly finite; $g_{(0,\infty)}(0,0)$ is readily seen to be bounded above by 
that same expectation, and~\eqref{h4} follows.



\section{Bounding the upper tail of an auxiliary random variable}
\label{frakbd}

\setcounter{equation}{0}

\begin{lema}\label{fbd} 
	Let $\o$ be a $BDP(p,q)$ with $p<\frac12$ starting from the equilibrium distribution, and let $T$ be a random variable, independent of $\o$, 
	distributed as the sum of $L^3$  iid exponential random variables of rate $\rho>0$, with $L\geq1$. Set $M=\max_{0\leq t\leq T}\o(t)$. Then
	\begin{equation}\label{fbd1}
		P\big(M>(\log L)^2\big)\leq e^{-c(\log L)^2}
	\end{equation}
	for all large enough $L$, where $c$ is a positive constant.
\end{lema}

\begin{proof}
	We may assume, using the upper bound on $\vf$, and resorting to standard large deviation estimates, that $N:=$ 
	the number of jumps of $\o$ up to $T$ is bounded above by $H:=\lceil c'L^3\rceil$, where $c'$ is a finite constant.
	
	Let now $W$ be a random variable distributed as the equilibrium distribution of $\o$, and, independently, consider $Y_1,Y_2,\ldots$,
	independent copies of the maximum of a discrete time simple random walk on $\Z$ with probability of jumping to the left 
	given by $p$ and starting at the origin.
	One now readily checks that $M\preceq W+(Y_1\vee\cdots\vee Y_H)$. Thus
	\begin{equation}\label{fbd2}
		P\big(M>(\log L)^2\big)\leq P\big(W>(\log L)^2/2\big)+H P\big(Y_1>(\log L)^2/2\big)
	\end{equation}
	plus a term that decays expoenentially in $L^3$ (which controls the probability that $N>H$), 
	and the well known exponential tails of the distributions of $W$ and $Y_1$ yield the result.
\end{proof}

\begin{observacao}\label{gen_init}
	Notice that it is enough to have the distribuion $W$ with an exponentially decaying tail for the argument to hold.
\end{observacao}

\begin{lema}\label{fbe} 
	Let $\o$ be a $BDP(p,q)$ with $p<\frac12$ starting from a distribution $\bar\mu$ with an exponentially decaying tail
	as in~\eqref{expdec}. 	Then, 
	there exist constants $C$ and $\beta'>0$ such that 
	\begin{equation}\label{fbe1}
		P\big(\o(t)>n\big)	\leq C e^{-\beta' n}
	\end{equation}
	for all $n$ (uniformly in $t$).
\end{lema}

\begin{proof}
	The strategy follows lines already pursued above.
	We consider a coupling of $\o$ to a copy $\o'$ starting from $\nu$, the equilibrium for the BDP, and independent till 
	first meeting, after which both processes coalesce.  Let $T$ and $T'$ denote the respective times of first visit to the origin
	by $\o$ and $\o'$. We also couple $\o$ to $Z$, a simple random walk, so that $\o=Z$ up to $T$. 
	We readily check that $T'$ (but also $T$) has an expoential tail. 
	
	The result will be readily implied from the following remarks.
	\begin{itemize}
		\item On $\{T\vee T'\leq t\}$, we have that $\o(t)=\o'(t)\sim\nu$;
		\item on $\{T>t\}$, we have that $\o(t)=Z_t\preceq \tilde Y=\max_{0\leq s<\infty}Z_s$;
		\item on $\{T\leq t<T'\leq n\}\cap\{N'\leq 2n\}$, we have that $\o(t)\leq\max_{0\leq s\leq T'}\o(s)\preceq W+(Y_1\vee\cdots\vee Y_{2n})$,
	\end{itemize}
	with $N', W,\,Y_1,\,Y_2\ldots$ as in the previous proof, with $N'$ defined similarly as $N$, except that $\o',\,T'$ replace $\o,\,T$, respectively.
	Thus
	\begin{eqnarray}
		P(\o(t)>n)	&\leq& \nu([n,\infty))+P(\tilde Y>n) + P(T'>n) + P(N'>2n)\nn\\\label{fbe2}&& + P(W> n/2) + 2n P(Y_1>n/2),
	\end{eqnarray}
	and the result follows.
\end{proof}



\subsection*{Acknowledgements}

LRF warmly thanks Elena Zhizhina for many enjoyable discussions during a visit of hers to São Paulo, where this project
originated.  He also thanks Hubert Lacoin for bringing up the issue addressed in Remark~\ref{vfzero}.




\begin{thebibliography}{99}
	
	\bibitem{AdHR}
	L.~Avena, F.~den Hollander, and F.~Redig, \emph{Law of large numbers for a
		class of random walks in dynamic random environments}, Electron. J. Probab.
	\textbf{16} (2011), no. 21, 587--617. 
	
	\bibitem{BZ06}
	Antar Bandyopadhyay and Ofer Zeitouni, \emph{Random walks in dynamic
		markovian random environment}, Alea \textbf{1} (2006), 205--224.
	
	\bibitem{BAC}
	G\'{e}rard Ben~Arous and Ji\v{r}\'{\i} \v{C}ern\'{y}, \emph{Scaling limit for
		trap models on {$\mathbb Z^d$}}, Ann. Probab. \textbf{35} (2007), no.~6,
	2356--2384. 
	
	\bibitem{BHT}
	Oriane Blondel, Marcelo~R. Hil\'{a}rio, and Augusto Teixeira, \emph{Random
		walks on dynamical random environments with nonuniform mixing}, Ann. Probab.
	\textbf{48} (2020), no.~4, 2014--2051. 
	
	\bibitem{BMP97}
	C.~Boldrighini, R.~A. Minlos, and A.~Pellegrinotti, \emph{Almost-sure central
		limit theorem for a {M}arkov model of random walk in dynamical random
		environment}, Probab. Theory Related Fields \textbf{109} (1997), no.~2,
	245--273. 
	
	\bibitem{BMPZ}
	C.~Boldrighini, R.~A. Minlos, A.~Pellegrinotti, and E.~A. Zhizhina,
	\emph{Continuous time random walk in dynamic random environment}, Markov
	Process. Related Fields \textbf{21} (2015), no.~4, 971--1004. 
	
	\bibitem{BPZ}
	C.~Boldrighini, A.~Pellegrinotti, and E.~A. Zhizhina, \emph{Regular and
		singular continuous time random walk in dynamic random environment}, Mosc.
	Math. J. \textbf{19} (2019), no.~1, 51--76. 
	
	\bibitem{BMP94}
	Carlo Boldrighini, Robert~A. Minlos, and Alessandro Pellegrinotti,
	\emph{Interacting random walk in a dynamical random environment. {II}.
		{E}nvironment from the point of view of the particle}, Ann. Inst. H.
	Poincar\'{e} Probab. Statist. \textbf{30} (1994), no.~4, 559--605.
	
	\bibitem{dHdS}
	F.~den Hollander and R.~S. dos Santos, \emph{Scaling of a random walk on a
		supercritical contact process}, Ann. Inst. Henri Poincar\'{e} Probab. Stat.
	\textbf{50} (2014), no.~4, 1276--1300. 
	
	\bibitem{DKL08}
	Dmitry Dolgopyat, Gerhard Keller, and Carlangelo Liverani, \emph{Random walk in
		markovian environment}, Ann. Probab. \textbf{36} (2008), no.~5, 1676--1710.
	
	\bibitem{RHG91}
	R.~Durrett, H.~Kesten, and G.~Lawler, \emph{Making money from fair games},
	pp.~255--267, Birkh{\"a}user Boston, Boston, MA, 1991.
	
	\bibitem{Em}
	D.~J. Emery, \emph{Limiting behaviour of the distributions of the maxima of
		partial sums of certain random walks}, J. Appl. Probability \textbf{9}
	(1972), 572--579. 
	
	\bibitem{FIN02}
	L.~R.~G. Fontes, M.~Isopi, and C.~M. Newman, \emph{Random walks with strongly
		inhomogeneous rates and singular diffusions: convergence, localization and
		aging in one dimension}, Ann. Probab. \textbf{30} (2002), no.~2, 579--604.
	
	\bibitem{HKT}
	Marcelo~R. Hil\'{a}rio, Daniel Kious, and Augusto Teixeira, \emph{Random walk
		on the simple symmetric exclusion process}, Comm. Math. Phys. \textbf{379}
	(2020), no.~1, 61--101. 
	
	\bibitem{HR}
	P.~L. Hsu and Herbert Robbins, \emph{Complete convergence and the law of large
		numbers}, Proc. Nat. Acad. Sci. U.S.A. \textbf{33} (1947), 25--31. 
	
	\bibitem{JP97}
	N.~James and Y.~Peres, \emph{Cutpoints and exchangeable events for random
		walks}, Theory of Probability \& Its Applications \textbf{41} (1997), no.~4,
	666--677.
	
	\bibitem{KMcG}
	Samuel Karlin and James McGregor, \emph{The classification of birth and death
		processes}, Trans. Amer. Math. Soc. \textbf{86} (1957), 366--400. 
	
	\bibitem{Lig}
	Thomas~M. Liggett, \emph{Interacting particle systems}, Grundlehren der
	mathematischen Wissenschaften [Fundamental Principles of Mathematical
	Sciences], vol. 276, Springer-Verlag, New York, 1985. 
	
	\bibitem{MV}
	Thomas Mountford and Maria~E. Vares, \emph{Random walks generated by
		equilibrium contact processes}, Electron. J. Probab. \textbf{20} (2015), no.
	3, 17. 
	
	\bibitem{RV13}
	Frank Redig and Florian V{\"o}llering, \emph{Random walks in dynamic random
		environments: a transference principle}, Ann. Probab. \textbf{41} (2013),
	no.~5, 3157--3180.
	
	\bibitem{Spi76}
	Frank Spitzer, \emph{Principles of random walk}, second ed., Graduate Texts in
	Mathematics, Vol. 34, Springer-Verlag, New York-Heidelberg, 1976.
	
	\bibitem{Y09}
	Atilla Yilmaz, \emph{Large deviations for random walk in a space-time product
		environment}, Ann. Probab. \textbf{37} (2009), no.~1, 189--205.
	
\end{thebibliography}



\end{document}